\documentclass[twoside]{article}
\usepackage[utf8]{inputenc}
\usepackage{graphicx}
\usepackage{latexsym}
\usepackage{amsmath}
\usepackage{amsfonts}
\usepackage{amssymb}
\usepackage{amsthm}
\usepackage{graphicx}
\usepackage[english]{babel}
\usepackage{xcolor}

\newcommand{\disp}{\displaystyle}

\newcommand{\fin}{\hfill\mbox{$\quad{}_{\Box}$}}
\newcommand{\fineq}{\vspace{-.75cm$\fin$}\par\bigskip}

\newtheorem{coro}{\bf \sffamily Corollary}
\newtheorem{theo}{\bf \sffamily Theorem}
\newtheorem{prop}{\bf \sffamily Proposition}
\newtheorem{rem}{\bf \sffamily Remark}
\newtheorem{lemma}{\bf \sffamily Lemma}

\begin{document}

\title{Compactly Supported Solutions near a Nehari Critical Value}
\author{ J.I. D\'{\i}az$^{1}$, J. Hern\'andez$^{1}$, and Y. Ilyasov$^{2}$ \\
{\small $^{1}$Instituto de Matem\'{a}tica Interdisciplinar, Universidad
Complutense de Madrid, 28040 Madrid, Spain }\\
{\small $^{2}$Institute of Mathematics with Computing Centre,}\\
{\small Ufa Federal Research Centre of the Russian Academy of Sciences,
450008 Ufa, Russia }\\
[1mm] {\small \texttt{jidiaz@ucm.es}, \texttt{jesus.hernande@telefonica.net}%
, \texttt{ilyasov02@gmail.com} }}
\date{}
\maketitle

\begin{abstract}
We study non-negative solutions of the indefinite sublinear Dirichlet
problem 
$$
\left \{\
\begin{array}{ll}
-\Delta u=\lambda u+m(x)|u|^{\alpha -1}u&\quad\hbox{ in $\Omega$},\\
u=0&\quad \hbox{ on $\partial \Omega$},
\end{array}
\right .
$$
where $\Omega \subset \mathbb{R}^{N}$ is a
smooth bounded domain, $0<\alpha <1$, $m\in L^{\infty }(\Omega )$ is a
sign-changing or non-positive weight, and $\lambda \in \mathbb{R}$.

Using variational and fibering methods on the Nehari manifold, we construct
two branches of non-negative solutions distinguished by the sign of the
second derivative of the energy along rays. The branch $u_{\lambda }^{+}$,
for which $E_{\lambda }^{\prime \prime }(u_{\lambda }^{+})>0,$ has negative
energy, whereas the branch $u_{\lambda }^{-}$, for which $E_{\lambda
}^{\prime \prime }(u_{\lambda }^{-})<0,$ has positive energy. Their
existence and asymptotic behavior are governed by the critical value 
\begin{equation*}
\lambda ^{\ast }=\inf \left\{\dfrac{\disp \int_{\Omega }|\nabla u|^{2}\,dx}{
\disp \int_{\Omega }|u|^{2}\,dx}:u\in H_{0}^{1}(\Omega )\setminus \{0\},\quad
\int_{\Omega }m(x)|u|^{\alpha +1}\,dx\geq 0\right\} .
\end{equation*}

In the degenerate case $m\leq 0$, when the zero region $\Omega ^{0}=\hbox{int%
}\{x\in \Omega :m(x)=0\}$ has positive measure, we identify $\lambda ^{\ast }$ with
the first Dirichlet eigenvalue $\lambda _{1,\Omega ^{0}}$. We show that the
positive-energy branch $u_{\lambda }^{-}$ tends to zero as $\lambda \uparrow
\lambda _{1,\Omega ^{0}}$, while its normalized profile converges to the
first eigenfunction in $\Omega ^{0}$.

We then combine uniform $L^{\infty }$-estimates with a local supersolution
argument to obtain dead-core formation. In particular, if the weight is
uniformly negative in a neighborhood of $\partial \Omega $, then solutions
on the positive-energy branch $u_{\lambda }^{-}$ have compact support for $%
\lambda $ sufficiently close to $\lambda _{1,\Omega ^{0}}$ from below. After obtaining explicit solutions to a one-dimensional degenerate problem, we
prove the convergence of the free boundaries (the boundary of the support of 
$u_{\lambda }^{-}$ to the boundary of $\Omega ^{0}$) as $\lambda \uparrow
\lambda _{1,\Omega ^{0}}.$ For sign-changing weights, solutions on the
negative-energy branch $u_{\lambda }^{+}$ also have compact support when $%
\lambda $ is sufficiently negative.
\end{abstract}



\section{Introduction}

Let $\Omega\subset\mathbb{R}^N$, $N\geq1$, be a bounded connected domain
with sufficiently smooth boundary. In this paper, we study the existence and
qualitative properties of non-negative solutions of the Dirichlet problem
\begin{equation}\label{Eq1}
\begin{cases}
	-\Delta u=\lambda u+m(x)|u|^{\alpha-1}u & \text{in }\Omega,\\
	u=0 & \text{on }\partial\Omega,
\end{cases}
\end{equation}
where $0<\alpha<1,\lambda\in\mathbb{R}, m\in L^\infty(\Omega)$.

The weight $m$ may change sign and may vanish on a nonempty open subset of
$\Omega$. Problem \eqref{Eq1} combines a linear spectral term with a sublinear
indefinite reaction. The non-Lipschitz character of the nonlinearity at the
origin allows non-negative solutions to develop a dead core, that is, a
nonempty open set on which the solution vanishes identically. Consequently,
besides strictly positive solutions, problem \eqref{Eq1} may possess flat
solutions satisfying $u>0$ in $\Omega$ and
$u=\partial u/\partial n=0$ on $\partial\Omega$, as well as solutions whose
support is compactly contained in $\Omega$. Such phenomena arise in models
of nonlinear diffusion, population dynamics, reaction--diffusion processes,
and porous media. We refer, for example, to
\cite{Fleming,Bandle,Cantrell} and the references therein. The literature
for $\alpha\geq1$ is extensive, but will not be discussed here, since our
interest is restricted to $0<\alpha<1$.

The variational analysis of elliptic problems with indefinite nonlinearities
has a long history; see, among others, \cite{AlamaTarantello1996,BerestyckiCapuzzoDolcettaNirenberg1995}.
The fibering method provides a natural framework for detecting and
 distinguishing variational branches; see \cite{DrabekPohozaev1997}. Related parameter-dependent variational ideas and nonlinear Rayleigh quotient methods for the analysis of nonlinear bifurcations and extreme parameter values of Nehari-type manifolds were developed in\cite{Ilyasov1997,Ilyasov2002,ilyaReil,IlyasKaye,Ilyasov2022}.

For the more general quasilinear problem
\[
-\Delta_pu
=
\lambda|u|^{p-2}u+a(x)|u|^{q-2}u,
\qquad 1<q<p,
\]
Kaufmann and Ramos \cite{Kaufmann-Ramos} established, among other results,
the existence of a negative-energy ground state below a constrained spectral
threshold and, under the condition
$\int_\Omega a\varphi_p^q\,dx<0$, the existence of a second
positive-energy solution between the first eigenvalue and that threshold.
They also studied uniqueness, positivity, and dead-core formation.
Subsequently, Bobkov and Tanaka \cite{Bobkov-Tanaka} gave a fine analysis of
the critical and supercritical spectral intervals for the same
$p$-Laplacian problem. In particular, they showed that, for genuinely
nonlinear principal parts and under the condition $p>2q$, non-negative
solutions may persist to the right of the constrained critical value.
This phenomenon cannot occur in the present semilinear sublinear case
$p=2$, $q=\alpha+1\in(1,2)$.

The semilinear problem \eqref{Eq1} was recently considered by D\'iaz,
Hern\'andez and Il'yasov \cite{DiazHernandezIlyasov2026}. That work pursued
a broad qualitative program involving non-negative and positive solutions,
bifurcation from infinity, uniqueness and stability questions, local
barriers, compactly supported solutions, and Pohozaev-type restrictions.
It also used the Nehari manifold to obtain two variationally distinguished
non-negative solutions in the interval between the first eigenvalue and a
constrained spectral threshold. Thus, neither the definition of the
threshold nor the basic existence of the two variational levels is claimed
here as a new result.

The purpose of the present paper is more focused. We isolate the critical
mechanism specific to weights having a nonempty open zero region and give a
self-contained treatment of the corresponding limiting branch. Our main
emphasis is placed on the degenerate case
\[
m\leq0,\qquad
\Omega^0=\operatorname{int}\{m=0\},\qquad
|\Omega^0|>0.
\]
We identify the constrained critical value with the first Dirichlet
eigenvalue of $\Omega^0$, prove that the positive-energy branch vanishes as
the parameter approaches this value from below, and identify its normalized
limit with the first eigenfunction of the zero region, extended by zero
outside $\Omega^0$. This spectral asymptotic information is then combined
with a quantitative local supersolution argument to prove the formation of a
dead core in a boundary ring whose width is uniform along the branch.
Consequently, the solutions have compact support for parameters sufficiently
close to the finite critical value.

We also analyze the opposite asymptotic regime $\lambda\to-\infty$ for
the negative-energy branch. A direct maximum estimate yields uniform
smallness of the solutions, and the same local dead-core mechanism then
produces a fixed zero ring near $\partial\Omega$. These results describe two
distinct routes to compact support: one is governed by the spectral geometry
of the zero set of the weight at a finite endpoint, whereas the other is
driven by uniform amplitude decay as the parameter tends to $-\infty$.

Accordingly, the contribution of this paper is not a further broad
classification of all positive and non-negative solutions of \eqref{Eq1}.
Rather, it is the rigorous identification and analysis of the chain
\[
\lambda^*=\lambda_{1,\Omega^0},
\qquad
u_\lambda^-\rightarrow0,
\qquad
\frac{u_\lambda^-}{\|u_\lambda^-\|_{H_0^1(\Omega)}}
\rightarrow
\widetilde{\varphi}_{1,\Omega^0},
\]
where $\widetilde{\varphi}_{1,\Omega^0}$ denotes the zero extension of
$\varphi_{1,\Omega^0}$ to $\Omega$, together with its consequence: a uniform
boundary dead core and compact support of the corresponding variational
solutions. We use the term ``dead core'' even though, in our setting, the
zero region need not be located at the center of the domain.

We use the notation
\[
\Omega^+:=\{x\in\Omega:m(x)>0\},\qquad
\Omega^-:=\{x\in\Omega:m(x)<0\},
\]
and
\[
Z_m:=\{x\in\Omega:m(x)=0\},\qquad
\Omega^0:=\operatorname{int}Z_m.
\]
Here and throughout the paper, $|A|$ denotes the Lebesgue measure of a
measurable set $A\subset\mathbb{R}^N$. Whenever one of the sets
$\Omega^+$, $\Omega^-$, or $\Omega^0$ is required to be nontrivial in the
measure-theoretic sense, this will be stated explicitly.

Whenever spectral quantities associated with $\Omega^0$ or
$\Omega^0\cup\Omega^+$ are used, we assume that the corresponding set is a
nonempty connected Lipschitz domain. Whenever spectral quantities associated
with $\Omega^0$ are involved, we also assume that
$\Omega^0\Subset\Omega$.

Whenever an argument identifies a function vanishing a.e. on
$\{m\neq0\}$ with an element of $H_0^1(\Omega^0)$, we further assume that
\[
Z_m=\overline{\Omega^0}
\quad\text{up to a set of Lebesgue measure zero}.
\]
In particular, in the degenerate case $m\leq0$, this amounts to assuming
that $m<0$ a.e. in $\Omega\setminus\overline{\Omega^0}$.

\noindent
We denote by $\lambda_{1,\Omega}$ and $\varphi_{1,\Omega}$ the first
Dirichlet eigenvalue and a corresponding positive eigenfunction of
$-\Delta$ in $\Omega$. Similarly, $\lambda_{1,\Omega^0}$ and
$\varphi_{1,\Omega^0}$ denote the first Dirichlet eigenpair in $\Omega^0$.
When regarded as a function on $\Omega$, $\varphi_{1,\Omega^0}$ is extended
by zero outside $\Omega^0$, and this extension is denoted by
$\widetilde{\varphi}_{1,\Omega^0}$.

The variational functional associated with \eqref{Eq1} is
\[
E_\lambda(u)
=
\frac12\int_\Omega\bigl(|\nabla u|^2-\lambda u^2\bigr)\,dx
-\frac{1}{\alpha+1}\int_\Omega m(x)|u|^{\alpha+1}\,dx,
\qquad
u\in H_0^1(\Omega).
\]
We also write
\[
H_\lambda(u)
:=
\int_\Omega\bigl(|\nabla u|^2-\lambda u^2\bigr)\,dx,
\qquad
M(u)
:=
\int_\Omega m(x)|u|^{\alpha+1}\,dx.
\]
The corresponding Nehari manifold is
\[
\mathcal N_\lambda
:=
\{u\in H_0^1(\Omega)\setminus\{0\}:H_\lambda(u)=M(u)\}.
\]
Its decomposition according to the sign of the second derivative of the
fibering map leads to two different families of solutions, denoted by
$u_\lambda^+$ and $u_\lambda^-$.

A central role is played by the Nehari critical value
\begin{equation}
\lambda^*
:=
\inf\left\{
\frac{\displaystyle\int_\Omega|\nabla u|^2\,dx}
{\displaystyle\int_\Omega|u|^2\,dx}
:
u\in H_0^1(\Omega)\setminus\{0\},\quad M(u)\geq0
\right\}.
\label{Ouyang}
\end{equation}

The value $\lambda^*$ is an extreme parameter value associated with the
Nehari constraint. Extreme values of this type and the behavior of solution
branches near them are naturally described within the nonlinear Rayleigh
quotient framework; see
\cite{Ilyasov2002,ilyaReil,IlyasKaye,Ilyasov2022}
and the references therein. We adopt the convention
$\lambda^*=+\infty$ if $|\Omega^+\cup\Omega^0|=0$.
By the variational characterization of the first Dirichlet eigenvalue,
$0<\lambda_{1,\Omega}\leq\lambda^*$, and
$\lambda^*<+\infty$ whenever $|\Omega^+\cup\Omega^0|>0$.
The position of $\lambda^*$ relative to $\lambda_{1,\Omega}$ is determined
by the sign of the weight along the first eigenfunction:
\begin{equation}
\lambda^*>\lambda_{1,\Omega}
\quad\Leftrightarrow\quad
\int_\Omega m(x)\varphi_{1,\Omega}^{\alpha+1}\,dx<0,
\label{eq:lambda-star-characterization}
\end{equation}
whereas
\begin{equation}
\lambda^*=\lambda_{1,\Omega}
\quad\Leftrightarrow\quad
\int_\Omega m(x)\varphi_{1,\Omega}^{\alpha+1}\,dx\geq0.
\label{eq:lambda-star-equality}
\end{equation}

For completeness, and in order to fix the particular variational selections
used later, we first recall the two solution levels below $\lambda^*$,
together with the properties needed in the sequel. The role of the next
theorem is preparatory: it identifies the precise branches whose asymptotic
profiles and associated supports are analyzed in the subsequent sections.

\begin{theo}
\label{thm1}
Assume that $0<\alpha<1$ and $m\in L^\infty(\Omega)$.

\begin{itemize}
\item[i)]
If $|\Omega^+|>0$, then, for every $\lambda<\lambda^*$, problem
\eqref{Eq1} possesses a non-negative weak solution
$u_\lambda^+\in H_0^1(\Omega)\cap C^{1,\gamma}(\overline{\Omega})$ for any
$\gamma\in(0,1)$, such that
\[
E_\lambda(u_\lambda^+)<0,
\qquad
E_\lambda''(u_\lambda^+)>0.
\]
Moreover, the minimum energy level $(-\infty,\lambda^*)\ni\lambda
\longmapsto E_\lambda(u_\lambda^+)$ is continuous and strictly decreasing, and
\[
E_\lambda(u_\lambda^+)\longrightarrow0
\qquad\text{as }\lambda\to-\infty.
\]
\item[ii)]
Suppose that
\[
\int_\Omega m(x)\varphi_{1,\Omega}^{\alpha+1}\,dx<0,
\qquad
|\Omega^+\cup\Omega^0|>0.
\]
Then
$0<\lambda_{1,\Omega}<\lambda^*<+\infty$.
For every $\lambda\in(\lambda_{1,\Omega},\lambda^*)$, problem
\eqref{Eq1} possesses a non-negative weak solution
$u_\lambda^-\in H_0^1(\Omega)\cap C^{1,\gamma}(\overline{\Omega})$ for any
$\gamma\in(0,1)$, such that
\[
E_\lambda(u_\lambda^-)>0,
\qquad
E_\lambda''(u_\lambda^-)<0.
\]
Moreover, the corresponding minimum energy level $(\lambda_{1,\Omega},\lambda^*)\ni\lambda
\longmapsto E_\lambda(u_\lambda^-)$ is continuous and strictly decreasing, and
\[
E_\lambda(u_\lambda^-)\longrightarrow+\infty
\qquad\text{as }\lambda\downarrow\lambda_{1,\Omega}.
\]
If, in addition, condition
\begin{equation}
\lambda_{1,\Omega^0}>
\lambda_{1,\Omega^0\cup\Omega^+}
\tag{M}
\label{HypoFirstEigenvalues}
\end{equation}
holds, then
\[
E_\lambda(u_\lambda^-)\longrightarrow0
\qquad\text{as }\lambda\uparrow\lambda^*.
\]
\end{itemize}
In the degenerate case $m\leq0$ a.e. in $\Omega$, assume that
$\Omega^0=\operatorname{int}\{m=0\}$ is a nonempty connected Lipschitz
domain compactly contained in $\Omega$ and that
$m<0$ a.e. in $\Omega\setminus\overline{\Omega^0}$.
Then
\[
\lambda^*=\lambda_{1,\Omega^0},
\qquad
E_\lambda(u_\lambda^-)\longrightarrow0
\quad\text{as }\lambda\uparrow\lambda_{1,\Omega^0},
\]
without assuming condition {\rm (M)}.

\end{theo}

The proof of Theorem~\ref{thm1} will be given later. We point out that the
basic existence assertions in parts \textnormal{i)} and
\textnormal{ii)} belong to the variational framework developed for
subhomogeneous indefinite problems in
\cite{Kaufmann-Ramos,Bobkov-Tanaka,DiazHernandezIlyasov2026}.
The continuity and strict monotonicity of the selected minimum energy levels
are proved in the present manuscript.
In the nondegenerate case, the endpoint relation
$E_\lambda(u_\lambda^-)\to0$ as $\lambda\uparrow\lambda^*$ follows under
condition {\rm (M)}. In the degenerate case $m\leq0$, however, the same
endpoint limit remains valid without {\rm (M)}, as proved below.

\begin{figure}[tph]
\begin{center}
	\includegraphics[width=7cm]{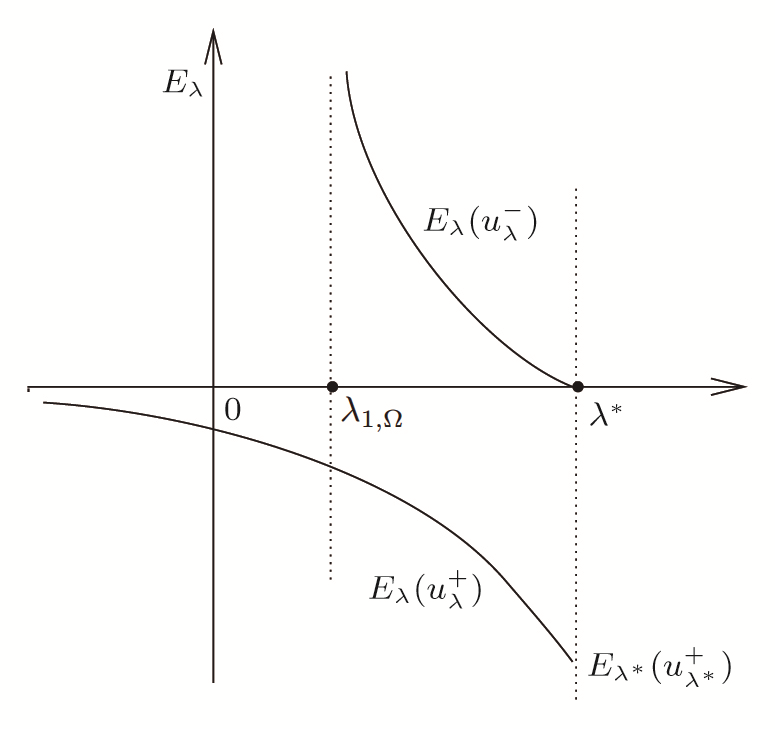}\\[0pt]
\end{center}
\caption{Qualitative behavior of the energy levels
	$E_\protect\lambda(u_\protect\lambda^+)$ and
	$E_\protect\lambda(u_\protect\lambda^-)$ when
	$\protect\int_\Omega
	m(x)\protect\varphi_{1,\Omega}^{\protect\alpha+1}\,dx<0$.}
\end{figure}

The branches $u_\lambda^+$ and $u_\lambda^-$ have different variational
origins. The first corresponds to a stable critical point, in the sense of
\cite{Dupaigne}, of the fibering map and has negative energy. The second
corresponds to an unstable critical point and has positive energy. Note that
the stable energy branch is decreasing. This does not contradict the
well-known result \cite{Crandel} that a decreasing bifurcation branch is
unstable when represented in terms of the $L^p$ norm, because here the
branch represents the energy level rather than the $L^p$ norm.

The corresponding two-level structure is known for the general
subhomogeneous indefinite problem; see \cite{Kaufmann-Ramos} and
\cite[Theorem~1.7]{Bobkov-Tanaka}. In the present semilinear case,
$\lambda^*$ is the terminal value of the positive-energy branch. Our
subsequent analysis concerns what happens to this branch as the endpoint is
approached and how its smallness affects the geometry of its support.

The degenerate case $m\leq0$ and $|\Omega^0|>0$ has a particularly
transparent spectral structure. Under the standing assumptions,
\[
\lambda^*=\lambda_{1,\Omega^0}.
\]
Moreover, since $m<0$ a.e. in
$\Omega\setminus\overline{\Omega^0}$,
\[
\int_\Omega m(x)\varphi_{1,\Omega}^{\alpha+1}\,dx<0,
\qquad
\lambda_{1,\Omega}
<
\lambda^*
=
\lambda_{1,\Omega^0}.
\]
The solution $u_\lambda^-$ then converges to zero as
$\lambda\uparrow\lambda_{1,\Omega^0}$, whereas its normalized profile
converges to $\widetilde{\varphi}_{1,\Omega^0}$. This vanishing property is
the starting point for our compact-support analysis.

To obtain a dead core near $\partial\Omega$, negativity of the weight merely
at boundary points is not sufficient. We require a full boundary ring on
which the weight is uniformly negative. For $\rho>0$, set
\[
\mathcal C_\rho
:=
\{x\in\Omega:\operatorname{dist}(x,\partial\Omega)<\rho\}.
\]
Our first compact-support theorem concerns the degenerate case $m\leq0$.

\begin{theo}
\label{thm3}
	Assume that $0<\alpha<1$, $m\in L^\infty(\Omega)$, $m\leq0$ a.e. in
	$\Omega$, and that
	\[
	\Omega^0:=\operatorname{int}\{x\in\Omega:m(x)=0\}
	\]
	is a nonempty connected $C^{1,1}$ domain compactly contained in $\Omega$, $|\Omega^0|>0$. Assume also that
	$m<0$ a.e. in $\Omega\setminus\overline{\Omega^0}$.
Suppose that there exist $\rho_0>0$ and $m_{\rho_0}^->0$ such that
\begin{equation}\label{eq:boundary-collar}
	\mathcal C_{\rho_0}\subset\Omega^-
\end{equation}
and
\begin{equation}\label{eq:negative-boundary-weight}
	-m(x)\geq m_{\rho_0}^-
	\qquad\text{for a.e. }x\in\mathcal C_{\rho_0}.
\end{equation}
Then there exists $\varepsilon>0$ such that, for every $\lambda\in
(\lambda_{1,\Omega^0}-\varepsilon,\lambda_{1,\Omega^0}),$
the non-negative weak solution $u_\lambda^-$ given by
Theorem~\ref{thm:Main2} has compact support in $\Omega$. More precisely, there exists $\rho\in(0,\rho_0)$, independent of $\lambda$
sufficiently close to $\lambda_{1,\Omega^0}$, such that
$u_\lambda^-=0$ in $\mathcal C_\rho$. Consequently,
\[
\operatorname{supp}u_\lambda^-
\subset
\{x\in\Omega:\operatorname{dist}(x,\partial\Omega)\geq\rho\}
\Subset\Omega.
\]
\end{theo}

\begin{figure}[tph]
\begin{center}
	\includegraphics[width=7cm]{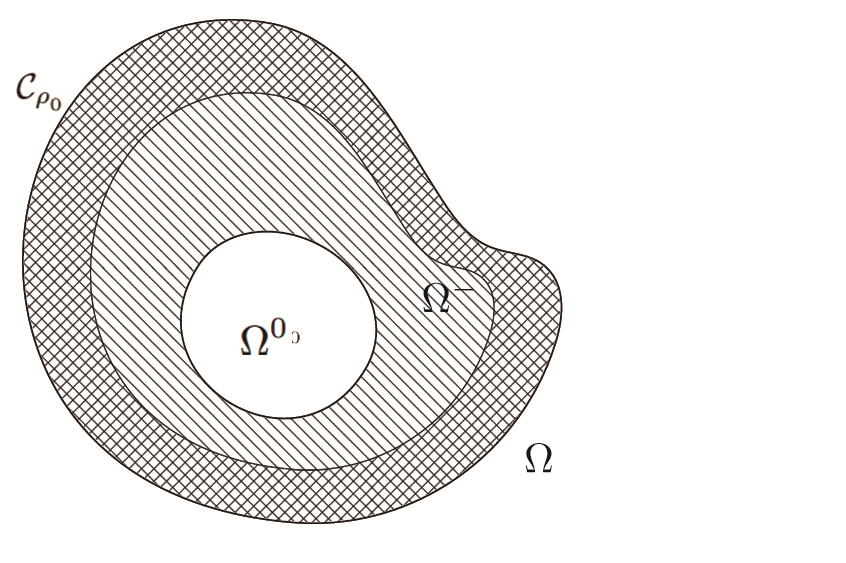}\\[0pt]
\end{center}
\caption{A uniformly negative boundary ring
	$\mathcal C_{\protect\rho_0}\subset\Omega^-$ and the zero region
	$\Omega^0=\operatorname{int}\{x\in\Omega:m(x)=0\}$.}
\label{fig:negative-boundary-collar}
\end{figure}

The preceding result concerns the upper endpoint of the branch
$u_\lambda^-$. A complementary phenomenon occurs on the branch
$u_\lambda^+$ as $\lambda\to-\infty$. In this regime, a direct maximum
estimate shows that
\[
\|u_\lambda^+\|_{L^\infty(\Omega)}
\leq
\left(
\frac{\|m^+\|_{L^\infty(\Omega)}}{|\lambda|}
\right)^{\frac{1}{1-\alpha}},
\]
and therefore
\[
\|u_\lambda^+\|_{L^\infty(\Omega)}
\longrightarrow0
\qquad\text{as }\lambda\to-\infty.
\]
Combined with the local dead-core criterion, this yields compactly supported
solutions for sufficiently negative values of the parameter.

\begin{theo}
\label{thm4}
Assume that $0<\alpha<1$, $|\Omega^+|>0$, $|\Omega^-|>0$, and
$m\in L^\infty(\Omega)$. Suppose that there exist $\rho_0>0$ and $m_{\rho_0}^-$ such that
\begin{equation}\label{eq:negative-boundary-collar-thm4}
	\mathcal C_{\rho_0}\subset\Omega^-
\end{equation}
and
\begin{equation}\label{eq:uniform-negative-boundary-thm4}
	-m(x)\geq m_{\rho_0}^-
	\qquad\text{for a.e. }x\in\mathcal C_{\rho_0}.
\end{equation}
Let $u_\lambda^+$ be the non-negative weak solution of \eqref{Eq1} given by
Theorem~\ref{thm1}. Then there exists $\sigma<0$ such that, for every
$\lambda<\sigma$, the solution $u_\lambda^+$ has compact support in
$\Omega$.

More precisely, there exists $\rho\in(0,\rho_0)$, independent of all
sufficiently negative $\lambda$, such that
$u_\lambda^+=0$ in $\mathcal C_\rho$. Consequently,
\[
\operatorname{supp}u_\lambda^+
\subset
\{x\in\Omega:\operatorname{dist}(x,\partial\Omega)\geq\rho\}
\Subset\Omega.
\]
\end{theo}

Theorems~\ref{thm3} and~\ref{thm4} describe two different mechanisms leading
to compactly supported solutions. In Theorem~\ref{thm3}, the amplitude tends
to zero as the parameter approaches the finite spectral threshold
$\lambda_{1,\Omega^0}$ from below. In Theorem~\ref{thm4}, the amplitude
tends to zero as $\lambda\to-\infty$. In both cases, uniform negativity of
the weight in a boundary ring allows a local barrier argument to transform
smallness of the solution into the formation of a dead core close to
$\partial\Omega$.

The distinction from the results of \cite{Kaufmann-Ramos} is that the weight
is kept fixed: the dead core is produced by the vanishing of a distinguished
variational branch rather than by increasing the negative part of the
weight. The distinction from \cite{Bobkov-Tanaka} is equally clear. Their
principal new phenomenon concerns continuation into the supercritical
interval for $p>2q$, whereas our problem belongs to the semilinear regime
$p=2$, where such continuation does not occur. Our results instead identify
the internal spectral object governing the terminal behavior and use it to
derive uniform localization of the support.

The proofs combine nonlinear Rayleigh quotients, minimization on the Nehari
manifold, compactness arguments, spectral properties of the zero region
$\Omega^0$, uniform $L^\infty$ estimates, and local supersolutions of
compact-support type. The paper is organized as follows. We first introduce
the variational setting and analyze the Nehari critical value $\lambda^*$.
We then establish the existence and asymptotic behavior of the two selected
solution branches. Next, we identify $\lambda^*$ with
$\lambda_{1,\Omega^0}$ in the degenerate non-positive case and determine the
normalized limiting profile. We develop a local dead-core criterion and
apply it to obtain compactly supported solutions both near the finite
spectral endpoint and for sufficiently negative values of $\lambda$.
We then prove convergence of the free boundaries, i.e. the boundaries of
the supports of $u_\lambda^-$, to $\partial\Omega^0$ as
$\lambda\uparrow\lambda_{1,\Omega^0}$. For sign-changing weights, solutions
on the negative-energy branch $u_\lambda^+$ also have compact support when
$\lambda$ is sufficiently negative. Finally, we establish free-boundary
convergence in several additional cases.

\section{Preliminaries}

We write \(H_0^1:=H_0^1(\Omega)\), endowed with the norm
\[
\|u\|:=\left(\int_\Omega |\nabla u|^2\,dx\right)^{1/2},
\]
which, by the Poincar\'e inequality, is equivalent to the standard
\(H^1\)-norm on \(H_0^1(\Omega)\). Throughout the paper, integrals without
an explicitly indicated domain are understood to be taken over \(\Omega\).
We denote by \(L^p(\Omega)\), \(1<p<\infty\), the usual Lebesgue spaces,
with norm \(\|\cdot\|_p\). For \(u\in H_0^1(\Omega)\), set
\[
H_\lambda(u):=\int_\Omega\bigl(|\nabla u|^2-\lambda |u|^2\bigr)\,dx,
\qquad
M(u):=\int_\Omega m(x)|u|^{\alpha+1}\,dx,
\]
and define the associated energy functional by
\[
E_\lambda(u):=\frac12 H_\lambda(u)-\frac{1}{\alpha+1}M(u).
\]

A weak solution of \eqref{Eq1} is a critical point of \(E_\lambda\) in
\(H_0^1(\Omega)\). We construct solutions by minimizing \(E_\lambda\) on
the two components of the Nehari manifold:
\begin{align}
\widehat E_\lambda^+
&:=\inf\{E_\lambda(u):u\in\mathcal N_\lambda^+\},
\label{MinConv1}\\
\widehat E_\lambda^-
&:=\inf\{E_\lambda(u):u\in\mathcal N_\lambda^-\}.
\label{MinConv2}
\end{align}
Here
\[
\mathcal N_\lambda^\pm
:=
\{u\in H_0^1(\Omega)\setminus\{0\}:
E_\lambda'(u)=0,\ \pm E_\lambda''(u)>0\},
\]
while
\[
\mathcal N_\lambda
:=
\{u\in H_0^1(\Omega)\setminus\{0\}:E_\lambda'(u)=0\}.
\]

For \(F\in C^1(H_0^1(\Omega))\), we use the notation
\[
F'(u):=DF(u)[u]
=\left.\frac{d}{dt}F(tu)\right|_{t=1}.
\]
Thus
\[
E_\lambda'(u)=H_\lambda(u)-M(u),
\qquad
E_\lambda''(u)=H_\lambda(u)-\alpha M(u),
\]
and, in particular,
\[
u\in\mathcal N_\lambda
\quad\Longleftrightarrow\quad
H_\lambda(u)=M(u).
\]

\begin{lemma}
\label{Lemm3.3}
Let \(u\in H_0^1(\Omega)\) be a non-negative weak solution of \eqref{Eq1}.
Then \(u\in W^{2,p}(\Omega)\) for every \(p<\infty\), and consequently
\(u\in C^{1,\gamma}(\overline\Omega)\) for every \(\gamma\in(0,1)\).
If, in addition, \(m\) is locally H\"older continuous, then
\(u\in C^2(\Omega)\).
\end{lemma}

\proof
A bootstrap argument, as in \cite{Drabeck-Nicolosi}, gives
\(u\in L^\infty(\Omega)\). Hence
\(\lambda u+m(x)u^\alpha\in L^p(\Omega)\) for every \(p<\infty\), and
standard elliptic regularity yields \(u\in W^{2,p}(\Omega)\).
Sobolev embedding then gives
\(u\in C^{1,\gamma}(\overline\Omega)\) for every \(\gamma\in(0,1)\).
If \(m\) is locally H\"older continuous, interior Schauder regularity
implies \(u\in C^2(\Omega)\). See also
\cite{AgmonDN,Gilb-Trud} and \cite[Appendix~B]{Struwe}. $\fin$

\section{Solutions at the Nehari critical value}

We now study the critical value
\begin{equation}\label{OuyangRepeated}
\lambda^*
:=
\inf\left\{
\frac{\displaystyle\int_\Omega |\nabla u|^2\,dx}
{\displaystyle\int_\Omega |u|^2\,dx}
:
u\in H_0^1(\Omega)\setminus\{0\},\quad M(u)\geq0
\right\}.
\end{equation}
By definition, \(0<\lambda_{1,\Omega}\leq\lambda^*\), and
\(\lambda^*<+\infty\) whenever \(|\Omega^0\cup\Omega^+|>0\).
If \(|\Omega^0\cup\Omega^+|=0\), we set \(\lambda^*:=+\infty\).

The position of \(\lambda^*\) relative to the first eigenvalue is
characterized by
\[
\lambda^*>\lambda_{1,\Omega}
\quad\Leftrightarrow\quad
\int_\Omega m(x)\varphi_{1,\Omega}^{\alpha+1}\,dx<0,
\]
whereas
\[
\lambda^*=\lambda_{1,\Omega}
\quad\leftrightarrow\quad
\int_\Omega m(x)\varphi_{1,\Omega}^{\alpha+1}\,dx\geq0.
\]

The definition of \(\lambda^*\) also gives, for every \(\lambda<\lambda^*\),
\begin{equation}\label{eq:21}
M(u)\geq0 \quad\Rightarrow\quad H_\lambda(u)>0,
\end{equation}
or, equivalently,
\begin{equation}\label{eq:22}
H_\lambda(u)\leq0 \quad\Rightarrow\quad M(u)<0,
\end{equation}
for every \(u\in H_0^1(\Omega)\setminus\{0\}\).

We shall use the following assumption.

\begin{description}
\item[\textbf{(M)}]
If \(|\Omega^0|>0\), then \(\Omega^0\cup\Omega^+\) is a domain and
\begin{equation}
	\lambda_{1,\Omega^0\cup\Omega^+}
	<
	\lambda_{1,\Omega^0}.
	\tag{M}
	\label{Condition M}
\end{equation}
\end{description}

Whenever \(\lambda_{1,\Omega^0\cup\Omega^+}\) is used, we assume that
\(\Omega^0\cup\Omega^+\) is a nonempty bounded Lipschitz domain.

\begin{rem}[On condition {\rm (M)} and its geometric meaning]
Condition {\rm (M)} is essentially geometric. If \(\Omega^0\) is a proper
subdomain of the connected domain \(\Omega^0\cup\Omega^+\), then strict
domain monotonicity of the first Dirichlet eigenvalue gives $\lambda_{1,\Omega^0\cup\Omega^+}<\lambda_{1,\Omega^0}$. Thus, in this standard situation, the spectral inequality in {\rm (M)}
is automatic: the positive phase genuinely enlarges the zero phase.
The connectedness assumption is important. If the positive phase is
separated from \(\Omega^0\) by a region where \(m<0\), then
\(\Omega^0\cup\Omega^+\) may be disconnected. For instance, in one
dimension, if $\Omega^0=(-1,1),\qquad I_\varepsilon=(2,2+\varepsilon)$,
with a negative region between them, then
\[
\lambda_{1,\Omega^0}=\frac{\pi^2}{4},
\qquad
\lambda_{1,I_\varepsilon}=\frac{\pi^2}{\varepsilon^2}.
\]
For \(0<\varepsilon<2\), the added positive component has a larger first
eigenvalue and hence the lowest first eigenvalue of the
disconnected union. Condition {\rm (M)} excludes precisely this
geometric--spectral decoupling. See \cite{Henrot,Davies}.
\end{rem}

\begin{lemma}
\label{lem:FritzJohn}
Let
\[
\mathcal R(u):=
\frac{\displaystyle\int_\Omega|\nabla u|^2\,dx}
{\displaystyle\int_\Omega|u|^2\,dx},
\qquad
\mathcal A:=
\{u\in H_0^1(\Omega)\setminus\{0\}:M(u)\geq0\}.
\]
Suppose that \(\bar u\in\mathcal A\) is a local minimizer of
\(\mathcal R\) on \(\mathcal A\). Then there exist
\(\mu_0,\mu_1\geq0\), with \(\mu_0+\mu_1>0\), such that
\begin{equation}\label{eq:FritzJohn}
	\mu_0D\mathcal R(\bar u)=\mu_1DM(\bar u),
\end{equation}
and
\begin{equation}\label{eq:complementary-slackness}
	\mu_1M(\bar u)=0.
\end{equation}
If \(DM(\bar u)\neq0\), then \(\mu_0>0\), and consequently there exists
\(\mu\geq0\) such that
\begin{equation}
	D\mathcal R(\bar u)=\mu\,DM(\bar u),
	\qquad
	\mu M(\bar u)=0.
	\label{eq:KKT}
\end{equation}
Thus, under the constraint qualification \(DM(\bar u)\neq0\), the
Fritz John conditions \cite{John} reduce to the usual Kuhn--Tucker
conditions.
\end{lemma}

\proof
Since \(\mathcal R\) and \(M\) are continuously Fr\'echet differentiable on
\(H_0^1(\Omega)\setminus\{0\}\), the Fritz John theorem applied to the
constraint \(M(u)\geq0\) gives \eqref{eq:FritzJohn} together with
\eqref{eq:complementary-slackness}; see, e.g., \cite{Boyd}.

If \(DM(\bar u)\neq0\) and \(\mu_0=0\), then
\eqref{eq:FritzJohn} and \(\mu_0+\mu_1>0\) imply
\(\mu_1>0\) and hence \(DM(\bar u)=0\), a contradiction.
Thus \(\mu_0>0\). Dividing \eqref{eq:FritzJohn} by \(\mu_0\) and setting
\(\mu=\mu_1/\mu_0\) gives \eqref{eq:KKT}.$\fin$

\begin{theo}
\label{thm:ExisExtr}
Assume that $|\Omega^0\cup\Omega^+|>0$, that the standing domain
regularity assumptions hold, and
\[
\int_\Omega m(x)\varphi_{1,\Omega}^{\alpha+1}\,dx<0.
\]
Then the following assertions hold.	
\begin{itemize}
\item[i)]
Problem \eqref{OuyangRepeated} has a nonzero non-negative minimizer
\(\overline u\in H_0^1(\Omega)\). Moreover, $H_{\lambda^*}(\overline u)=0, M(\overline u)=0$.
\item[ii)]
If \(|\Omega^+|=0\), then
\(\lambda^*=\lambda_{1,\Omega^0}\), and \(\overline u\) is a positive
multiple of \(\varphi_{1,\Omega^0}\), extended by zero outside
\(\Omega^0\).

\item[iii)]
If, in addition, condition \textbf{(M)} holds, then there exists
\(\tau>0\) such that \(\tau\overline u\) is a weak solution of
\eqref{Eq1} for \(\lambda=\lambda^*\), and $E_{\lambda^*}''(\tau\overline u)=0$.		
\end{itemize}
\end{theo}

\proof
We divide the proof into three steps.

\medskip
\noindent
$i)$
The admissible set in \eqref{OuyangRepeated} is nonempty, since the first
eigenfunction of \(-\Delta\) in \(\Omega^0\cup\Omega^+\), extended by
zero to \(\Omega\), satisfies
\(M(\varphi_{1,\Omega^0\cup\Omega^+})\geq0\). Hence
\begin{equation}\label{eq:ineqEig}
	\lambda^*\leq\lambda_{1,\Omega^0\cup\Omega^+}.
\end{equation}	
Let \((u_n)\) be a minimizing sequence normalized by
\(\|u_n\|_{L^2(\Omega)}=1\). Then
\(\int_\Omega|\nabla u_n|^2\,dx\to\lambda^*\), so \((u_n)\) is bounded
in \(H_0^1(\Omega)\). Up to a subsequence,
\[
u_n\rightharpoonup\overline u \quad\text{in }H_0^1(\Omega),
\qquad
u_n\to\overline u \quad\text{in }L^r(\Omega),\quad 1\leq r<2^*.
\]
In particular, \(\|\overline u\|_{L^2(\Omega)}=1\), hence
\(\overline u\neq0\). Since \(m\in L^\infty(\Omega)\),
\(M(u_n)\to M(\overline u)\), so \(M(\overline u)\geq0\).
Weak lower semicontinuity gives
\[
\mathcal R(\overline u)
\leq\liminf_{n\to\infty}\mathcal R(u_n)=\lambda^*.
\]
By the definition of \(\lambda^*\), equality holds. Thus
\(\overline u\) is a minimizer,
\(H_{\lambda^*}(\overline u)=0\), and, replacing it by
\(|\overline u|\), we may assume \(\overline u\geq0\).

It remains to show that \(M(\overline u)=0\).
By Lemma~\ref{lem:FritzJohn}, there exist
\(\mu_0,\mu_1\geq0\), \(\mu_0+\mu_1>0\), such that
\begin{equation}\label{eq:67}
	\mu_0D\mathcal R(\overline u)=\mu_1DM(\overline u).
\end{equation}
Since \(\mathcal R(\overline u)=\lambda^*\),
\[
D\mathcal R(\overline u)[v]
=
\frac{1}{\|\overline u\|_{L^2(\Omega)}^2}
DH_{\lambda^*}(\overline u)[v].
\]
Absorbing the positive factor into \(\mu_0\), we obtain
\begin{equation}\label{eq:67H}
	\mu_0DH_{\lambda^*}(\overline u)
	=
	\mu_1DM(\overline u).
\end{equation}
Equivalently,
\[
2\mu_0(-\Delta\overline u-\lambda^*\overline u)
=
(\alpha+1)\mu_1m(x)|\overline u|^{\alpha-1}\overline u
\]
in the weak sense.
By complementary slackness (see, e.g., \cite{Boyd}), \(\mu_1M(\overline u)=0\). If \(\mu_1=0\), then \(\mu_0>0\) and
\(-\Delta\overline u=\lambda^*\overline u\) in \(\Omega\).
Since \(\overline u\geq0\) and \(\overline u\not\equiv0\), the strong
maximum principle yields
\(\lambda^*=\lambda_{1,\Omega}\) and
\(\overline u=c\varphi_{1,\Omega}\), contradicting
\[
\int_\Omega m(x)\varphi_{1,\Omega}^{\alpha+1}\,dx<0.
\]
Hence \(\mu_1>0\), and therefore \(M(\overline u)=0\).
\par
\noindent
$ii)$
If \(|\Omega^+|=0\), then \(m\leq0\) a.e. in \(\Omega\).
Since \(M(\overline u)=0\), we have
\(\overline u=0\) a.e. on \(\{m<0\}\).
By the standing zero-set and regularity assumptions,
\(\overline u\in H_0^1(\Omega^0)\). Hence
\[
\lambda_{1,\Omega^0}
\leq
\frac{\displaystyle\int_{\Omega^0}|\nabla\overline u|^2\,dx}
{\displaystyle\int_{\Omega^0}|\overline u|^2\,dx}
=
\lambda^*.
\]
Conversely, \(\varphi_{1,\Omega^0}\) is admissible in
\eqref{OuyangRepeated}, so
\(\lambda^*\leq\lambda_{1,\Omega^0}\).
Thus \(\lambda^*=\lambda_{1,\Omega^0}\), and simplicity of the first
eigenvalue gives
\(\overline u=c\varphi_{1,\Omega^0}\) for some \(c>0\).

\medskip
\noindent
$iii)$
We first show that \(\mu_0>0\). Otherwise \(\mu_1>0\), and
\eqref{eq:67H} gives \(DM(\overline u)=0\), hence
\(\overline u=0\) a.e. on \(\{m\neq0\}\).
If \(|\Omega^0|=0\), this contradicts \(\overline u\neq0\).
If \(|\Omega^0|>0\), the standing zero-set and regularity assumptions give
\(\overline u\in H_0^1(\Omega^0)\), and
\(H_{\lambda^*}(\overline u)=0\) implies $\lambda_{1,\Omega^0}
\leq\lambda^*
\leq\lambda_{1,\Omega^0\cup\Omega^+}$, contrary to condition \textbf{(M)}. Thus \(\mu_0>0\).
Since also \(\mu_1>0\), set
\[
\tau
:=
\left(
\frac{2\mu_0}{(\alpha+1)\mu_1}
\right)^{\frac{1}{1-\alpha}}.
\]
Then \eqref{eq:67H} implies that \(\tau\overline u\) is a weak solution
of \eqref{Eq1} for \(\lambda=\lambda^*\).
Finally,
\(H_{\lambda^*}(\overline u)=M(\overline u)=0\), and homogeneity yields
\(H_{\lambda^*}(\tau\overline u)=M(\tau\overline u)=0\). Therefore $E_{\lambda^*}''(\tau\overline u)=0$. This completes the proof.$\fin$

\section{The case $\protect\lambda<\protect\lambda^*$}
Let
\[
S:=\{v\in H_0^1(\Omega):\|v\|=1\}.
\]
We introduce
\[
\Theta^+:=\{v\in H_0^1(\Omega)\setminus\{0\}:M(v)>0\},
\]
and, for $\lambda>\lambda_{1,\Omega}$,
\[
\Theta_\lambda^-:=
\{v\in H_0^1(\Omega)\setminus\{0\}:H_\lambda(v)<0\}.
\]
Clearly, $\Theta^+\neq\emptyset$ if $|\Omega^+|>0$, while
$\Theta_\lambda^-\neq\emptyset$ for every $\lambda>\lambda_{1,\Omega}$,
since
\[
H_\lambda(\varphi_{1,\Omega})
=
(\lambda_{1,\Omega}-\lambda)
\int_\Omega\varphi_{1,\Omega}^2\,dx<0.
\]

If $v\in\Theta^+$ and $\lambda<\lambda^*$, then \eqref{eq:21} gives
$H_\lambda(v)>0$. Similarly, if
$v\in\Theta_\lambda^-$ and
$\lambda_{1,\Omega}<\lambda<\lambda^*$, then~\eqref{eq:22} gives
$M(v)<0$.

In either case, the fibering equation
\begin{equation}\label{3.6}
\phi_v'(t)=tH_\lambda(v)-t^\alpha M(v)=0
\end{equation}
has the unique positive solution
\begin{equation}\label{eq:t}
t_\lambda(v)
=
\left(\frac{M(v)}{H_\lambda(v)}\right)^{\frac{1}{1-\alpha}},
\end{equation}
where the quotient is positive. Hence
\begin{equation}\label{eq:NehFib}
\mathcal N_\lambda^\pm
=
\{t_\lambda(v)v:v\in\Theta_\lambda^\pm\cap S\},
\end{equation}
with $\Theta_\lambda^+:=\Theta^+$. Set
\[
c_\alpha:=\frac{1-\alpha}{2(1+\alpha)}.
\]
For $v\in\Theta^+$, define
\begin{equation}\label{j2}
J_\lambda^+(v)
:=
E_\lambda(t_\lambda(v)v)
=
-c_\alpha
\frac{M(v)^{\frac{2}{1-\alpha}}}
{H_\lambda(v)^{\frac{1+\alpha}{1-\alpha}}},
\qquad
\lambda<\lambda^*,
\end{equation}
whereas, for $v\in\Theta_\lambda^-$,
\begin{equation}\label{j22}
J_\lambda^-(v)
:=
E_\lambda(t_\lambda(v)v)
=
c_\alpha
\frac{|M(v)|^{\frac{2}{1-\alpha}}}
{|H_\lambda(v)|^{\frac{1+\alpha}{1-\alpha}}},
\qquad
\lambda_{1,\Omega}<\lambda<\lambda^*.
\end{equation}
Both functionals are homogeneous of degree zero. We therefore consider
\begin{align}
\widehat J_\lambda^+
&:=
\inf_{v\in\Theta^+\cap S}J_\lambda^+(v),
\qquad
\lambda<\lambda^*,
\label{jmin1}\\
\widehat J_\lambda^-
&:=
\inf_{v\in\Theta_\lambda^-\cap S}J_\lambda^-(v),
\qquad
\lambda_{1,\Omega}<\lambda<\lambda^*.
\label{jmin2}
\end{align}
If $|\Omega^+|>0$, then $\widehat J_\lambda^+<0$.
Moreover, if
$\int_\Omega m(x)\varphi_{1,\Omega}^{\alpha+1}\,dx<0$, then
$\widehat J_\lambda^-<+\infty$ for every
$\lambda\in(\lambda_{1,\Omega},\lambda^*)$.

\begin{prop}
\label{prop:crJE2}
\
\begin{itemize}
\item[i)]
If $|\Omega^+|>0$, $\lambda<\lambda^*$, and
$v_\lambda^+\in\Theta^+\cap S$ minimizes \eqref{jmin1}, then
$u_\lambda^+:=t_\lambda(v_\lambda^+)v_\lambda^+$ is a non-negative weak
solution of \eqref{Eq1}. Moreover,
\[
E_\lambda(u_\lambda^+)
=
\widehat E_\lambda^+
=
\widehat J_\lambda^+<0,
\qquad
E_\lambda''(u_\lambda^+)>0.
\]
\item[ii)]
Assume that $|\Omega^0\cup\Omega^+|>0$ and
\[
\int_\Omega m(x)\varphi_{1,\Omega}^{\alpha+1}\,dx<0.
\]
If $\lambda\in(\lambda_{1,\Omega},\lambda^*)$ and
$v_\lambda^-\in\Theta_\lambda^-\cap S$ minimizes \eqref{jmin2}, then
$u_\lambda^-:=t_\lambda(v_\lambda^-)v_\lambda^-$ is a non-negative weak
solution of \eqref{Eq1}. Moreover,
\[
E_\lambda(u_\lambda^-)
=
\widehat E_\lambda^-
=
\widehat J_\lambda^->0,
\qquad
E_\lambda''(u_\lambda^-)<0.
\]
\end{itemize}
\end{prop}

\proof
The proof follows from the fibering representation \eqref{eq:NehFib} and
the standard natural-constraint argument for the Nehari manifold; see,
e.g., \cite{DiazHernandezIlyasov2026}. Since the minimizers may be replaced
by their absolute values, the resulting weak solutions are non-negative. $\fin$

\begin{lemma}
\label{ExisHomMin}
The following assertions hold.
\begin{itemize}
\item[i)]
If $|\Omega^+|>0$, then for every $\lambda<\lambda^*$ there exists a
minimizer $v_\lambda^+\in\Theta^+\cap S$ of \eqref{jmin1}.

\item[ii)]
Assume that $|\Omega^0\cup\Omega^+|>0$ and
\[
\int_\Omega m(x)\varphi_{1,\Omega}^{\alpha+1}\,dx<0.
\]
Then, for every $\lambda\in(\lambda_{1,\Omega},\lambda^*)$, there exists
a minimizer $v_\lambda^-\in\Theta_\lambda^-\cap S$ of \eqref{jmin2}.
\end{itemize}
\end{lemma}
\proof
$i)$
Fix $\lambda<\lambda^*$ and let
$(v_n)\subset\Theta^+\cap S$ be a minimizing sequence for \eqref{jmin1}.
Thus $\|v_n\|=1$, $M(v_n)>0$, and
\[
J_\lambda^+(v_n)\longrightarrow\widehat J_\lambda^+<0.
\]
Up to a subsequence, there exists $v_\lambda^+\in H_0^1(\Omega)$ such that
\[
v_n\rightharpoonup v_\lambda^+
\quad\text{in }H_0^1(\Omega),
\qquad
v_n\to v_\lambda^+
\quad\text{in }L^q(\Omega),\quad 1\leq q<2^*.
\]
In particular, $M(v_n)\to M(v_\lambda^+)$.	
We first obtain a uniform lower bound for $H_\lambda(v_n)$. Since
$M(v_n)>0$, each $v_n$ is admissible in the definition of $\lambda^*$,
and hence
\[
\lambda^*
\leq
\frac{\displaystyle\int_\Omega|\nabla v_n|^2\,dx}
{\displaystyle\int_\Omega|v_n|^2\,dx}.
\]
Because $\|v_n\|=1$, we have $\|v_n\|_2^2\leq1/\lambda^*$.
Consequently,
\[
H_\lambda(v_n)
\geq
\begin{cases}
	1-\lambda/\lambda^*, & 0\leq\lambda<\lambda^*,\\
	1, & \lambda<0.
\end{cases}
\]
Thus there exists $c_\lambda>0$, independent of $n$, such that
\begin{equation}\label{eq:H-positive-lower-bound}
	H_\lambda(v_n)\geq c_\lambda>0.
\end{equation}	
Suppose that $v_\lambda^+=0$. Then
$M(v_n)\to0$, and \eqref{j2} together with
\eqref{eq:H-positive-lower-bound} yields
$J_\lambda^+(v_n)\to0$, contradicting
$\widehat J_\lambda^+<0$. Hence $v_\lambda^+\neq0$.	
Since $M(v_n)>0$, we have $M(v_\lambda^+)\geq0$. If
$M(v_\lambda^+)=0$, the same argument again gives
$J_\lambda^+(v_n)\to0$, a contradiction. Therefore
$M(v_\lambda^+)>0$, so $v_\lambda^+\in\Theta^+$.	
By weak lower semicontinuity and strong $L^2$ convergence,
\[
H_\lambda(v_\lambda^+)
\leq
\liminf_{n\to\infty}H_\lambda(v_n).
\]
Since $M(v_n)\to M(v_\lambda^+)>0$, formula \eqref{j2} gives
\[
J_\lambda^+(v_\lambda^+)
\leq
\liminf_{n\to\infty}J_\lambda^+(v_n)
=
\widehat J_\lambda^+.
\]
If $\|v_\lambda^+\|<1$, normalize by setting
\[
\widetilde v_\lambda^+
:=
\frac{v_\lambda^+}{\|v_\lambda^+\|}.
\]
Since $J_\lambda^+$ is homogeneous of degree zero,
$\widetilde v_\lambda^+\in\Theta^+\cap S$ and
$J_\lambda^+(\widetilde v_\lambda^+)=J_\lambda^+(v_\lambda^+)$.
Hence
\[
\widehat J_\lambda^+
\leq
J_\lambda^+(\widetilde v_\lambda^+)
\leq
\widehat J_\lambda^+.
\]
Thus $\widetilde v_\lambda^+$ is a minimizer of \eqref{jmin1}.
Renaming it $v_\lambda^+$ proves $i)$.

\medskip
\noindent
$ii)$
The proof is analogous. Let
$(v_n)\subset\Theta_\lambda^-\cap S$ be a minimizing sequence.
Since $H_\lambda(v_n)<0$,
$\|v_n\|_2^2>1/\lambda$, and therefore its weak limit
$\bar v$ is nonzero. Strong $L^{\alpha+1}$ convergence gives
$M(v_n)\to M(\bar v)$.

If $M(\bar v)=0$, weak lower semicontinuity gives
$H_\lambda(\bar v)\leq0$, contradicting \eqref{eq:21}, since
$\lambda<\lambda^*$ and $\bar v\neq0$. Thus $M(\bar v)<0$.
Since the minimizing sequence has bounded energy and
$M(v_n)$ stays away from zero, \eqref{j22} shows that
$H_\lambda(v_n)$ stays uniformly away from zero, and consequently
$H_\lambda(\bar v)<0$. Formula \eqref{j22}, weak lower semicontinuity,
and the homogeneity of $J_\lambda^-$ then yield, after normalization,
a minimizer $v_\lambda^-\in\Theta_\lambda^-\cap S$ of \eqref{jmin2}.$\fin$

Combining Proposition~\ref{prop:crJE2} and Lemma~\ref{ExisHomMin}, we obtain
the following existence result.

\begin{coro}
\label{lem:extr}
Assume that $|\Omega^0\cup\Omega^+|>0$ and
\[
\int_\Omega m(x)\varphi_{1,\Omega}^{\alpha+1}\,dx<0.
\]	
\begin{itemize}
\item[i)]
If $|\Omega^+|>0$, then for every $\lambda<\lambda^*$ there exists a
non-negative weak solution $u_\lambda^+$ of \eqref{Eq1} such that
\[
E_\lambda(u_\lambda^+)<0,
\qquad
E_\lambda''(u_\lambda^+)>0.
\]
\item[ii)]
For every $\lambda\in(\lambda_{1,\Omega},\lambda^*)$, there exists a
non-negative weak solution $u_\lambda^-$ of \eqref{Eq1} such that
\[
E_\lambda(u_\lambda^-)>0,
\qquad
E_\lambda''(u_\lambda^-)<0.
\]
\end{itemize}
\end{coro}

\begin{rem}
By Lemma~\ref{Lemm3.3},
$u_\lambda^-\in W^{2,p}(\Omega)$ for every $p<\infty$.
\end{rem}

\subsection{Dependence on the parameter $\protect\lambda$}

The next result is based on the comparison argument introduced in
\cite{Ilyasov1997}.

\begin{prop}
\label{prop:IneqParam}
Let $\lambda_0$ belong to a compact subinterval of
$(\lambda_{1,\Omega},\lambda^*)$ for the branch $u_\lambda^-$, or of
$(-\infty,\lambda^*)$ for the branch $u_\lambda^+$. Then there exist
$\delta>0$ and
$r_i^\pm(\lambda,\lambda_0)=o(|\lambda-\lambda_0|)$, $i=1,2$, such that,
whenever $0<\lambda-\lambda_0<\delta$,
\begin{align}
	-\frac{\lambda-\lambda_0}{2}\|u_\lambda^-\|_2^2+r_1^-(\lambda,\lambda_0)
	&\leq E_\lambda(u_\lambda^-)-E_{\lambda_0}(u_{\lambda_0}^-)
	\notag\\
	&\leq -\frac{\lambda-\lambda_0}{2}\|u_{\lambda_0}^-\|_2^2
	+r_2^-(\lambda,\lambda_0).
	\label{eq:j-j}
\end{align}
and
\begin{align}
	-\frac{\lambda-\lambda_0}{2}\|u_\lambda^+\|_2^2+r_1^+(\lambda,\lambda_0)
	&\leq E_\lambda(u_\lambda^+)-E_{\lambda_0}(u_{\lambda_0}^+)
	\notag\\
	&\leq -\frac{\lambda-\lambda_0}{2}\|u_{\lambda_0}^+\|_2^2
	+r_2^+(\lambda,\lambda_0).
	\label{eq:j-jj2}
\end{align}
Moreover, for every fixed
$\lambda_0\in(\lambda_{1,\Omega},\lambda^*)$, there exist
$\delta,c_0,C_0>0$ such that, whenever $|\lambda-\lambda_0|<\delta$,
\[
|H_\lambda(u_\lambda^-)|\geq c_0,
\qquad
\|u_\lambda^-\|\leq C_0.
\]	
\end{prop}

\proof
We give the proof for $u_\lambda^-$; the argument for $u_\lambda^+$ is
analogous.	Fix $\lambda_0\in(\lambda_{1,\Omega},\lambda^*)$ and choose
\[
K:=[\lambda_0-\eta,\lambda_0+\eta]
\Subset(\lambda_{1,\Omega},\lambda^*)
\]
for some $\eta>0$. For each $\lambda\in K$, let
$v_\lambda^-\in\Theta_\lambda^-\cap S$ be a minimizer of \eqref{jmin2}
and write
$u_\lambda^-=t_\lambda(v_\lambda^-)v_\lambda^-$.

\medskip
\noindent
\textit{Step 1: uniform estimates.}
There exists $c_M>0$ such that
\begin{equation}\label{eq:M-uniform-negative}
	M(v)\leq-c_M<0
\end{equation}
for every $\lambda\in K$ and every $v\in S$ with $H_\lambda(v)\leq0$.
Indeed, otherwise there would exist $\lambda_n\in K$ and $v_n\in S$ such
that
$H_{\lambda_n}(v_n)\leq0$ and $M(v_n)\to0$. Up to a subsequence,
$\lambda_n\to\bar\lambda\in K$,
$v_n\rightharpoonup\bar v$ in $H_0^1(\Omega)$, and
$v_n\to\bar v$ in $L^q(\Omega)$ for $1\leq q<2^*$.
Since
\[
1-\lambda_n\|v_n\|_2^2
=
H_{\lambda_n}(v_n)\leq0,
\]
we have $\|v_n\|_2^2\geq1/\lambda_n\geq1/\max K$, so
$\bar v\neq0$. Moreover, $M(\bar v)=0$ and
\[
H_{\bar\lambda}(\bar v)
\leq
\liminf_{n\to\infty}H_{\lambda_n}(v_n)
\leq0,
\]
contradicting \eqref{eq:21}. Thus \eqref{eq:M-uniform-negative} holds. The minimum levels are uniformly bounded above on $K$. Indeed,
$M(\varphi_{1,\Omega})<0$, while
\[
H_\lambda(\varphi_{1,\Omega})
=
(\lambda_{1,\Omega}-\lambda)
\|\varphi_{1,\Omega}\|_2^2<0
\]
for all $\lambda\in K$. Hence
$w:=\varphi_{1,\Omega}/\|\varphi_{1,\Omega}\|$
belongs to $\Theta_\lambda^-\cap S$, and
$\widehat J_\lambda^-\leq J_\lambda^-(w)\leq C_K$.
Using \eqref{eq:M-uniform-negative} and \eqref{j22}, we obtain
\begin{equation}\label{eq:H-uniform-negative-normalized}
	|H_\lambda(v_\lambda^-)|\geq c_H>0,
	\qquad \lambda\in K.
\end{equation}
Since $v_\lambda^-\in S$, the values $M(v_\lambda^-)$ are uniformly
bounded. Thus \eqref{eq:t} and
\eqref{eq:H-uniform-negative-normalized} imply
\begin{equation}\label{eq:u-uniform-bounded}
	\|u_\lambda^-\|
	=
	t_\lambda(v_\lambda^-)
	\leq C,
	\qquad \lambda\in K.
\end{equation}
Moreover, the lower bound \eqref{eq:M-uniform-negative} and the boundedness
of $|H_\lambda(v_\lambda^-)|$ show that
$t_\lambda(v_\lambda^-)$ is bounded away from zero. Therefore
\begin{equation}\label{eq:H-uniform-negative-solution}
	|H_\lambda(u_\lambda^-)|\geq c_H>0,
	\qquad \lambda\in K.
\end{equation}

\medskip
\noindent
\textit{Step 2: differentiability of the reduced energy.}
For $v\neq0$ and parameters $\mu$ such that
$M(v)H_\mu(v)>0$, define
\[
\mathcal J(\mu,v):=E_\mu(t_\mu(v)v).
\]
Since $t_\mu(v)v\in\mathcal N_\mu$, differentiation gives
\begin{equation}\label{eq:reduced-energy-derivative}
	\frac{\partial\mathcal J}{\partial\mu}(\mu,v)
	=
	-\frac12 t_\mu(v)^2\|v\|_2^2.
\end{equation}
Furthermore,
\[
\frac{\partial t_\mu(v)}{\partial\mu}
=
\frac{t_\mu(v)\|v\|_2^2}
{(1-\alpha)H_\mu(v)},
\]
and hence
\begin{equation}\label{eq:reduced-energy-second-derivative}
	\frac{\partial^2\mathcal J}{\partial\mu^2}(\mu,v)
	=
	-
	\frac{t_\mu(v)^2\|v\|_2^4}
	{(1-\alpha)H_\mu(v)}.
\end{equation}	
By Step 1, for $\lambda,\lambda_0\in K$ sufficiently close, the quantities
$t_\mu(v_\lambda^-)$ and $\|v_\lambda^-\|_2$ remain uniformly bounded for
$\mu$ between $\lambda_0$ and $\lambda$, while
$|H_\mu(v_\lambda^-)|\geq c_H/2$. Hence the second derivative above is
uniformly bounded, and Taylor's formula gives
\begin{equation}\label{eq:Taylor-vlambda}
	\mathcal J(\lambda_0,v_\lambda^-)
	=
	\mathcal J(\lambda,v_\lambda^-)
	+
	\frac{\lambda-\lambda_0}{2}\|u_\lambda^-\|_2^2
	+
	R_1^-(\lambda,\lambda_0),
\end{equation}
and
\begin{equation}\label{eq:Taylor-vlambda0}
	\mathcal J(\lambda,v_{\lambda_0}^-)
	=
	\mathcal J(\lambda_0,v_{\lambda_0}^-)
	-
	\frac{\lambda-\lambda_0}{2}\|u_{\lambda_0}^-\|_2^2
	+
	R_2^-(\lambda,\lambda_0),
\end{equation}
where
$|R_i^-(\lambda,\lambda_0)|
\leq C|\lambda-\lambda_0|^2$, $i=1,2$.

\medskip
\noindent
\textit{Step 3: comparison of the minimum levels.}
Let $0<\lambda-\lambda_0<\delta$, with $\delta$ sufficiently small.
By \eqref{eq:H-uniform-negative-normalized},
$v_\lambda^-\in\Theta_{\lambda_0}^-$ and
$v_{\lambda_0}^-\in\Theta_\lambda^-$.
Therefore, the minimizing properties at $\lambda_0$ and $\lambda$,
combined with \eqref{eq:Taylor-vlambda} and
\eqref{eq:Taylor-vlambda0}, yield
\[
E_\lambda(u_\lambda^-)
-
E_{\lambda_0}(u_{\lambda_0}^-)
\geq
-\frac{\lambda-\lambda_0}{2}\|u_\lambda^-\|_2^2
-
R_1^-(\lambda,\lambda_0)
\]
and
\[
E_\lambda(u_\lambda^-)
-
E_{\lambda_0}(u_{\lambda_0}^-)
\leq
-\frac{\lambda-\lambda_0}{2}\|u_{\lambda_0}^-\|_2^2
+
R_2^-(\lambda,\lambda_0).
\]
Setting
\[
r_1^-:=-R_1^-,
\qquad
r_2^-:=R_2^-,
\]
we obtain \eqref{eq:j-j}, with
$r_i^-(\lambda,\lambda_0)=O(|\lambda-\lambda_0|^2)
=o(|\lambda-\lambda_0|)$.	
For $u_\lambda^+$, choose
$K\Subset(-\infty,\lambda^*)$. Since $M(v)>0$ for
$v\in\Theta^+\cap S$, the definition of $\lambda^*$ yields
$H_\lambda(v)\geq c_K>0$ uniformly for $\lambda\in K$.
A fixed admissible direction in $\Theta^+\cap S$, together with the
minimizing property, prevents $M(v_\lambda^+)$ from tending to zero.
Hence $t_\lambda(v_\lambda^+)$ is bounded above and below on $K$, and the
same Taylor argument applies. Since $\Theta^+$ is independent of $\lambda$,
the comparison is even simpler and gives \eqref{eq:j-jj2}, with $r_i^+(\lambda,\lambda_0)
=
O(|\lambda-\lambda_0|^2)
=
o(|\lambda-\lambda_0|).\fin$

\begin{coro}
\label{cor:monot}
The maps
\[
(-\infty,\lambda^*)\ni\lambda
\mapsto E_\lambda(u_\lambda^+)
\]
and
\[
(\lambda_{1,\Omega},\lambda^*)\ni\lambda
\mapsto E_\lambda(u_\lambda^-)
\]
are continuous and strictly decreasing.
\end{coro}

\proof
Let $\lambda>\lambda_0$. Proposition~\ref{prop:IneqParam} gives
\[
E_\lambda(u_\lambda^\pm)
-
E_{\lambda_0}(u_{\lambda_0}^\pm)
\leq
-\frac{\lambda-\lambda_0}{2}\|u_{\lambda_0}^\pm\|_2^2
+
o(\lambda-\lambda_0).
\]
Since $u_{\lambda_0}^\pm\neq0$, the right-hand side is negative for
$\lambda-\lambda_0>0$ sufficiently small. Thus the energy levels are
locally strictly decreasing, hence strictly decreasing on their whole
domains by subdivision of compact parameter intervals. Continuity follows
from the two-sided estimates in Proposition~\ref{prop:IneqParam}.$\fin$

\begin{rem}
The maps $\lambda\mapsto E_\lambda(u_\lambda^\pm)$ are single-valued because
they represent minimum energy levels. The minimizers themselves, and hence
the quantities $\|u_\lambda^\pm\|_2$, need not be unique.
\end{rem}

\subsection{Asymptotic behavior as $\protect\lambda\uparrow\protect\lambda^*$}

\begin{lemma}
\label{solextr}
Assume that
\(\int_\Omega m(x)\varphi_{1,\Omega}^{\alpha+1}\,dx<0\).
\begin{itemize}
\item[a)]
If $|\Omega^+|>0$, then there exists a non-negative weak solution
$u_{\lambda^*}^+$ of \eqref{Eq1} such that
\[
E_{\lambda^*}(u_{\lambda^*}^+)<0,
\qquad
E_{\lambda^*}''(u_{\lambda^*}^+)>0.
\]
Moreover,
\[
E_\lambda(u_\lambda^+)
\rightarrow E_{\lambda^*}(u_{\lambda^*}^+)
\qquad\text{as }\lambda\uparrow\lambda^*.
\]

\item[b)]
If $|\Omega^0\cup\Omega^+|>0$ and condition \textbf{(M)} holds, then
\[
E_\lambda(u_\lambda^-)\rightarrow0
\qquad\text{as }\lambda\uparrow\lambda^*.
\]
\end{itemize}
\end{lemma}

\proof
$a)$
By Corollary~\ref{cor:monot},
$\lambda\mapsto E_\lambda(u_\lambda^+)$ is strictly decreasing on
$(-\infty,\lambda^*)$. Hence
\[
\overline E_{\lambda^*}^+
:=
\lim_{\lambda\uparrow\lambda^*}E_\lambda(u_\lambda^+)
\]
exists in $[-\infty,0)$. We first prove that $\overline E_{\lambda^*}^+>-\infty$. Suppose otherwise that, for some $\lambda_n\uparrow\lambda^*$,
$E_{\lambda_n}(u_{\lambda_n}^+)\to-\infty$.
Since $u_{\lambda_n}^+\in\mathcal N_{\lambda_n}^+$,
\[
H_{\lambda_n}(u_{\lambda_n}^+)
=
M(u_{\lambda_n}^+)>0,
\]
and
\begin{equation}\label{eq:E-plus-Nehari}
	E_{\lambda_n}(u_{\lambda_n}^+)
	=
	-c_\alpha H_{\lambda_n}(u_{\lambda_n}^+)
	=
	-c_\alpha M(u_{\lambda_n}^+),
	\qquad
	c_\alpha=\frac{1-\alpha}{2(1+\alpha)}.
\end{equation}
Thus $M(u_{\lambda_n}^+)\to+\infty$, and therefore
$\|u_{\lambda_n}^+\|\to+\infty$.
Set
\[
\rho_n:=\|u_{\lambda_n}^+\|,
\qquad
v_n:=\frac{u_{\lambda_n}^+}{\rho_n}.
\]
Then $\|v_n\|=1$ and $v_n\geq0$. Up to a subsequence,
\[
v_n\rightharpoonup\bar v \quad\text{in }H_0^1(\Omega),
\qquad
v_n\to\bar v \quad\text{in }L^q(\Omega),\quad 1\leq q<2^*.
\]
Dividing the Nehari identity by $\rho_n^2$ gives
\[
H_{\lambda_n}(v_n)=\rho_n^{\alpha-1}M(v_n)\longrightarrow0.
\]
Hence $1-\lambda_n\|v_n\|_2^2\to0$, and consequently
$\|\bar v\|_2^2=1/\lambda^*>0$. Dividing the equation for $u_{\lambda_n}^+$ by $\rho_n$ and passing to
the limit yields
\[
-\Delta\bar v=\lambda^*\bar v\qquad\text{in }\Omega.
\]
Since $\bar v\geq0$ and $\bar v\not\equiv0$, the strong maximum principle
implies $\lambda^*=\lambda_{1,\Omega}$, contrary to
\[
\int_\Omega m(x)\varphi_{1,\Omega}^{\alpha+1}\,dx<0.
\]
Thus $\overline E_{\lambda^*}^+>-\infty$.
We next prove compactness as $\lambda\uparrow\lambda^*$. For any sequence
$\lambda_n\uparrow\lambda^*$, the family $(u_{\lambda_n}^+)$ is bounded in
$H_0^1(\Omega)$; otherwise the same normalization argument would again
give $\lambda^*=\lambda_{1,\Omega}$. Hence, up to a subsequence,
\[
u_{\lambda_n}^+\rightharpoonup u_{\lambda^*}^+
\quad\text{in }H_0^1(\Omega),
\qquad
u_{\lambda_n}^+\to u_{\lambda^*}^+
\quad\text{in }L^q(\Omega),\quad 1\leq q<2^*.
\]
In particular,
$M(u_{\lambda_n}^+)\to M(u_{\lambda^*}^+)$.
The limit is nonzero. Indeed, for any fixed $\lambda_0<\lambda^*$,
monotonicity gives
\[
E_{\lambda_n}(u_{\lambda_n}^+)
\leq E_{\lambda_0}(u_{\lambda_0}^+)<0
\]
for large $n$. If $u_{\lambda^*}^+=0$, then
$M(u_{\lambda_n}^+)\to0$, and \eqref{eq:E-plus-Nehari} would imply
$E_{\lambda_n}(u_{\lambda_n}^+)\to0$, a contradiction.	
Passing to the limit in the weak formulation, using the strong
$L^{\alpha+1}$ convergence, shows that $u_{\lambda^*}^+$ is a
non-negative weak solution of \eqref{Eq1} for $\lambda=\lambda^*$.
Testing the limit equation by $u_{\lambda^*}^+$ gives
\[
H_{\lambda^*}(u_{\lambda^*}^+)=M(u_{\lambda^*}^+).
\]
Moreover, from the Nehari identities,
\[
\|u_{\lambda_n}^+\|^2
=
\lambda_n\|u_{\lambda_n}^+\|_2^2+M(u_{\lambda_n}^+)
\longrightarrow
\|u_{\lambda^*}^+\|^2.
\]
Together with weak convergence, this yields
\[
u_{\lambda_n}^+\longrightarrow u_{\lambda^*}^+
\qquad\text{strongly in }H_0^1(\Omega).
\]
Consequently,
\[
E_{\lambda_n}(u_{\lambda_n}^+)
\longrightarrow
E_{\lambda^*}(u_{\lambda^*}^+)
=
\overline E_{\lambda^*}^+<0.
\]
Hence
$H_{\lambda^*}(u_{\lambda^*}^+)=M(u_{\lambda^*}^+)>0$, and therefore
\[
E_{\lambda^*}''(u_{\lambda^*}^+)
=
(1-\alpha)M(u_{\lambda^*}^+)>0.
\]
It remains to verify the minimizing property at $\lambda^*$.
Let $w\in\mathcal N_{\lambda^*}^+$. Then
$H_{\lambda^*}(w)=M(w)>0$. For large $n$, set
\[
s_n
:=
\left(
\frac{M(w)}{H_{\lambda_n}(w)}
\right)^{\frac1{1-\alpha}}.
\]
Then $s_nw\in\mathcal N_{\lambda_n}^+$ and $s_n\to1$. By minimality,
\[
E_{\lambda_n}(u_{\lambda_n}^+)
\leq E_{\lambda_n}(s_nw).
\]
Passing to the limit gives
$E_{\lambda^*}(u_{\lambda^*}^+)\leq E_{\lambda^*}(w)$.
Thus $u_{\lambda^*}^+$ minimizes $E_{\lambda^*}$ on
$\mathcal N_{\lambda^*}^+$, and the limit is independent of the chosen
sequence. This proves $a)$.
\par
\noindent
$b)$
Let $u_*$ be the nonzero non-negative minimizer associated with
$\lambda^*$ in Theorem~\ref{thm:ExisExtr}. Under condition \textbf{(M)},
Theorem~\ref{thm:ExisExtr}(3) allows us, after a positive rescaling, to
assume that $u_*$ is a weak solution of~\eqref{Eq1} for $\lambda=\lambda^*$.
Moreover,
\[
H_{\lambda^*}(u_*)=M(u_*)=0.
\]
Let $\varphi_1=\varphi_{1,\Omega}>0$. Testing the equation for $u_*$ by
$\varphi_1$ gives
\[
\int_\Omega m(x)u_*^\alpha\varphi_1\,dx
=
(\lambda_{1,\Omega}-\lambda^*)
\int_\Omega u_*\varphi_1\,dx<0.
\]
Hence
\[
DH_{\lambda^*}(u_*)[\varphi_1]<0,
\qquad
DM(u_*)[\varphi_1]<0.
\]
For $v_\varepsilon:=u_*+\varepsilon\varphi_1$, we therefore have
\[
H_{\lambda^*}(v_\varepsilon)
=
-c_1\varepsilon+o(\varepsilon),
\qquad
M(v_\varepsilon)
=
-c_2\varepsilon+o(\varepsilon)
\]
for some $c_1,c_2>0$.	
Choose $\varepsilon(\lambda)=\sqrt{\lambda^*-\lambda}$. Since
\[
H_\lambda(v_\varepsilon)
=
H_{\lambda^*}(v_\varepsilon)
+
(\lambda^*-\lambda)\|v_\varepsilon\|_2^2,
\]
for $\lambda$ sufficiently close to $\lambda^*$,
\[
H_\lambda(v_{\varepsilon(\lambda)})<0,
\qquad
M(v_{\varepsilon(\lambda)})<0.
\]
Thus $v_{\varepsilon(\lambda)}$ is admissible for the negative Nehari
level. By minimality,
\[
0<E_\lambda(u_\lambda^-)
\leq
J_\lambda^-(v_{\varepsilon(\lambda)}).
\]
Using \eqref{j22} and the preceding expansions,
\[
J_\lambda^-(v_{\varepsilon(\lambda)})
\leq C\varepsilon(\lambda)\longrightarrow0.
\]
Therefore $E_\lambda(u_\lambda^-)\longrightarrow0\text{ as }\lambda\uparrow\lambda^*.\fin$
\par
\medskip
For every weak solution $u_\lambda$ of \eqref{Eq1},
\begin{equation}\label{eq:EHM}
E_\lambda(u_\lambda)
=
-\frac{1-\alpha}{2(1+\alpha)}H_\lambda(u_\lambda)
=
-\frac{1-\alpha}{2(1+\alpha)}M(u_\lambda).
\end{equation}

\begin{coro}
\label{cor:solextr}
Assume that
\(\int_\Omega m(x)\varphi_{1,\Omega}^{\alpha+1}\,dx<0\).
\begin{itemize}
\item[a)]
If $|\Omega^+|>0$, then, as $\lambda\uparrow\lambda^*$,
\[
H_\lambda(u_\lambda^+)
\rightarrow H_{\lambda^*}(u_{\lambda^*}^+),
\qquad
M(u_\lambda^+)
\rightarrow M(u_{\lambda^*}^+).
\]

\item[b)]
If $|\Omega^0\cup\Omega^+|>0$ and condition \textbf{(M)} holds, then
\[
H_\lambda(u_\lambda^-)\longrightarrow0,
\qquad
M(u_\lambda^-)\longrightarrow0
\qquad\text{as }\lambda\uparrow\lambda^*.
\]
\end{itemize}
\end{coro}

\begin{lemma}
\label{lem:corD}
Assume that $0<\alpha<1$ and $m\in L^\infty(\Omega)$.
\begin{itemize}
\item[a)]
If
\(\int_\Omega m(x)\varphi_{1,\Omega}^{\alpha+1}\,dx<0\), then, as
$\lambda\downarrow\lambda_{1,\Omega}$,
\[
E_\lambda(u_\lambda^-)\rightarrow+\infty,
\qquad
\|u_\lambda^-\|\rightarrow+\infty,
\]
and, after normalizing $\|\varphi_{1,\Omega}\|=1$,
\[
\frac{u_\lambda^-}{\|u_\lambda^-\|}
\longrightarrow\varphi_{1,\Omega}
\qquad\text{strongly in }H_0^1(\Omega).
\]	
\item[b)]
If $|\Omega^+|>0$, then, as $\lambda\to-\infty$,
\[
E_\lambda(u_\lambda^+)\rightarrow0,
\qquad
\|u_\lambda^+\|\rightarrow0.
\]
\end{itemize}
\end{lemma}

\proof
$a)$
Let $\lambda_n\downarrow\lambda_{1,\Omega}$ and set
\[
v_n:=\frac{u_{\lambda_n}^-}{\|u_{\lambda_n}^-\|}.
\]
Then $\|v_n\|=1$ and $H_{\lambda_n}(v_n)<0$. Up to a subsequence,
$v_n\rightharpoonup\bar v$ in $H_0^1(\Omega)$ and
$v_n\to\bar v$ in $L^q(\Omega)$ for $1\leq q<2^*$.	
Since $1<\lambda_n\|v_n\|_2^2$, passing to the limit gives
$H_{\lambda_{1,\Omega}}(\bar v)\leq0$. The variational characterization
of $\lambda_{1,\Omega}$ gives the reverse inequality; hence
$H_{\lambda_{1,\Omega}}(\bar v)=0$.
Thus $\bar v$ is proportional to $\varphi_{1,\Omega}$, and under the chosen
normalization
\[
v_n\longrightarrow\varphi_{1,\Omega}
\qquad\text{strongly in }H_0^1(\Omega).
\]
Since $M(\varphi_{1,\Omega})<0$, formula \eqref{eq:t} yields
$\|u_{\lambda_n}^-\|=t_{\lambda_n}(v_n)\to+\infty$, and
\eqref{j22} gives
$E_{\lambda_n}(u_{\lambda_n}^-)\to+\infty$.
Since $(\lambda_n)$ was arbitrary, the assertions hold for the whole family.
\par
\noindent
$b)$
By Corollary~\ref{cor:monot}, the limit
\[
\overline E_{-\infty}^+
:=
\lim_{\lambda\to-\infty}E_\lambda(u_\lambda^+)\leq0
\]
exists. From \eqref{eq:EHM},
$H_\lambda(u_\lambda^+)$ remains bounded as $\lambda\to-\infty$.
Since
\[
H_\lambda(u_\lambda^+)
=
\|u_\lambda^+\|^2
+
|\lambda|\|u_\lambda^+\|_2^2,
\]
both terms on the right are bounded, and hence
$\|u_\lambda^+\|_2\to0$. Since $\alpha+1<2$ and $\Omega$ is bounded,
$\|u_\lambda^+\|_{\alpha+1}\to0$, so
\[
|M(u_\lambda^+)|
\leq
\|m\|_\infty
\|u_\lambda^+\|_{\alpha+1}^{\alpha+1}
\longrightarrow0.
\]
By \eqref{eq:EHM},
$E_\lambda(u_\lambda^+)\to0$ and
$H_\lambda(u_\lambda^+)\to0$. Finally,
\[
0\leq\|u_\lambda^+\|^2
\leq H_\lambda(u_\lambda^+),
\]
and therefore $\|u_\lambda^+\|\to0.\fin$

\section{\label{subsection degenerate}Degenerate case: $m(x)\leq0$ and
$|\Omega^0|>0$}

\begin{theo}
\label{thm:Main2}
Assume that $0<\alpha<1$, $m\in L^\infty(\Omega)$, $m\leq0$ a.e. in
$\Omega$, and that
\[
\Omega^0:=\operatorname{int}\{x\in\Omega:m(x)=0\}
\]
is a nonempty connected $C^{1,1}$ domain compactly contained in$\Omega$. Assume also that
	$m<0$ a.e. in $\Omega\setminus\overline{\Omega^0}$.
Then the following assertions hold.
\begin{itemize}
\item[i)]
For every
$\lambda\in(\lambda_{1,\Omega},\lambda_{1,\Omega^0})$, problem
\eqref{Eq1} possesses a non-negative weak solution $u_\lambda^-$ such that
\[
E_\lambda(u_\lambda^-)>0,
\qquad
E_\lambda''(u_\lambda^-)<0,
\]
and
\[
E_\lambda(u_\lambda^-)\rightarrow0
\qquad\text{as }\lambda\uparrow\lambda_{1,\Omega^0}.
\]

\item[ii)]
If $\lambda\geq\lambda_{1,\Omega^0}$, then problem \eqref{Eq1} has no
positive weak solution.

\item[iii)]
As $\lambda\uparrow\lambda_{1,\Omega^0}$,
\[
\|u_\lambda^-\|\rightarrow0.
\]
Moreover, after normalizing $\|\varphi_{1,\Omega^0}\|=1$,
\[
\frac{u_\lambda^-}{\|u_\lambda^-\|}
\rightarrow
\widetilde{\varphi}_{1,\Omega^0}
\qquad\text{strongly in }H_0^1(\Omega),
\]
where $\widetilde{\varphi}_{1,\Omega^0}$ denotes the zero extension of
$\varphi_{1,\Omega^0}$ to $\Omega$.
\end{itemize}
\end{theo}
\begin{proof}
$i)$
Since $m<0$ a.e. in
$\Omega\setminus\overline{\Omega^0}$ and
$\varphi_{1,\Omega}>0$ in $\Omega$, we have
\[
\int_\Omega m(x)\varphi_{1,\Omega}^{\alpha+1}\,dx<0.
\]
By Theorem~\ref{thm:ExisExtr}, $\lambda^*=\lambda_{1,\Omega^0}$, and an extremal function corresponding to $\lambda^*$ is, up to a positive
constant, $\widetilde{\varphi}_{1,\Omega^0}$.
Hence Corollary~\ref{lem:extr} yields, for every
$\lambda\in(\lambda_{1,\Omega},\lambda_{1,\Omega^0})$, a non-negative weak
solution $u_\lambda^-$ satisfying
\[
E_\lambda(u_\lambda^-)>0,
\qquad
E_\lambda''(u_\lambda^-)<0.
\]
It remains to prove that
\[
E_\lambda(u_\lambda^-)\rightarrow0
\qquad
\text{as }\lambda\uparrow\lambda_{1,\Omega^0}.
\]
Set
\[
\lambda^*:=\lambda_{1,\Omega^0},
\qquad
\phi:=\widetilde{\varphi}_{1,\Omega^0}.
\]
Then
\[
H_{\lambda^*}(\phi)=0,
\qquad
M(\phi)=0.
\]
Consider the bilinear form
\[
B_*(v,w)
:=
\int_\Omega\nabla v\cdot\nabla w\,dx
-\lambda^*\int_\Omega vw\,dx.
\]
We claim that there exists $\zeta\in H_0^1(\Omega)$ such that
\begin{equation}\label{eq:degenerate-zeta}
	B_*(\phi,\zeta)<0,
	\qquad
	M(\zeta)<0.
\end{equation}
Indeed, the functional $\zeta\mapsto B_*(\phi,\zeta)$ cannot vanish
identically on $H_0^1(\Omega)$. Otherwise $\phi$ would satisfy
\[
-\Delta\phi=\lambda^*\phi
\qquad\text{weakly in }\Omega,
\]
and the strong maximum principle would imply $\phi>0$ in $\Omega$,
contrary to $\phi=0$ in
$\Omega\setminus\overline{\Omega^0}$.
Thus there exists $\zeta_0\in H_0^1(\Omega)$ with
$B_*(\phi,\zeta_0)<0$.	
Choose a nonzero
$\eta\in C_c^\infty(\Omega\setminus\overline{\Omega^0})$.
Since $m<0$ a.e. there, $ M(\eta)<0$, while, because $\phi=0$ on the support of $\eta$,
$B_*(\phi,\eta)=0$.
Hence, for $\zeta=\zeta_0+s\eta$ with $s>0$ sufficiently large,
$M(\zeta)<0$, whereas
$B_*(\phi,\zeta)=B_*(\phi,\zeta_0)<0$.
This proves \eqref{eq:degenerate-zeta}.
Let
\[
\delta:=\lambda^*-\lambda>0,
\qquad
\varepsilon_\lambda:=K\delta,
\qquad
w_\lambda:=\phi+\varepsilon_\lambda\zeta,
\]
where $K>0$ will be chosen below. Since $m=0$ on $\Omega^0$ and
$\phi=0$ outside $\Omega^0$,
\begin{equation}\label{eq:degenerate-Mw}
	M(w_\lambda)
	=
	K^{\alpha+1}\delta^{\alpha+1}M(\zeta)<0.
\end{equation}
Moreover,
\[
H_\lambda(w_\lambda)
=
2\varepsilon_\lambda B_*(\phi,\zeta)
+\varepsilon_\lambda^2H_{\lambda^*}(\zeta)
+\delta\|\phi+\varepsilon_\lambda\zeta\|_2^2,
\]
and therefore
\begin{equation}\label{eq:degenerate-Hw}
	H_\lambda(w_\lambda)
	=
	\delta
	\bigl[
	\|\phi\|_2^2+2K B_*(\phi,\zeta)
	\bigr]
	+O(\delta^2).
\end{equation}
Since $B_*(\phi,\zeta)<0$, we may choose $K$ sufficiently large so that
\[
\|\phi\|_2^2+2K B_*(\phi,\zeta)<0.
\]
Hence there exist $c>0$ and $\delta_0>0$ such that
\begin{equation}\label{eq:degenerate-Hnegative}
	H_\lambda(w_\lambda)\leq-c\delta<0
	\qquad
	\text{for }0<\delta<\delta_0.
\end{equation}
Thus $w_\lambda\in\Theta_\lambda^-$ for $\lambda$ sufficiently close to
$\lambda^*$.
By the minimizing property of $u_\lambda^-$ and the homogeneity of
$J_\lambda^-$,
\[
0<E_\lambda(u_\lambda^-)
\leq J_\lambda^-(w_\lambda).
\]
Using \eqref{j22}, \eqref{eq:degenerate-Mw}, and
\eqref{eq:degenerate-Hnegative}, we obtain
\[
J_\lambda^-(w_\lambda)
\leq
C
\frac{
	\delta^{2(\alpha+1)/(1-\alpha)}
}{
	\delta^{(1+\alpha)/(1-\alpha)}
}
=
C\delta^{(1+\alpha)/(1-\alpha)}.
\]
Therefore
\[
E_\lambda(u_\lambda^-)\rightarrow0
\qquad
\text{as }\lambda\uparrow\lambda_{1,\Omega^0}.
\]
\par
\noindent
$ii)$
Suppose that $u_\lambda>0$ is a weak solution of \eqref{Eq1} for some
$\lambda\geq\lambda_{1,\Omega^0}$. Since $m=0$ a.e. in $\Omega^0$,
\[
-\Delta u_\lambda=\lambda u_\lambda
\qquad\text{in }\Omega^0.
\]
For $\phi\in C_0^\infty(\Omega^0)$, Picone's inequality gives
\[
\int_{\Omega^0}|\nabla\phi|^2\,dx
\geq
\lambda\int_{\Omega^0}\phi^2\,dx.
\]
Taking the infimum over
$\phi\in H_0^1(\Omega^0)\setminus\{0\}$ yields
$\lambda\leq\lambda_{1,\Omega^0}$, and hence
$\lambda=\lambda_{1,\Omega^0}$.	
To exclude equality, multiply
$-\Delta u_\lambda=\lambda_{1,\Omega^0}u_\lambda$
by $\varphi_{1,\Omega^0}$ and use Green's identity:
\[
\int_{\partial\Omega^0}
u_\lambda
\frac{\partial\varphi_{1,\Omega^0}}{\partial n}\,dS=0.
\]
By the Hopf lemma,
$\partial_n\varphi_{1,\Omega^0}<0$ on $\partial\Omega^0$, whereas
$u_\lambda>0$ there since $\Omega^0\Subset\Omega$.
The integral is therefore strictly negative, a contradiction.	
\medskip
\noindent
\textup{(iii)}
Let
\[
\lambda_n\uparrow\lambda_{1,\Omega^0},
\qquad
u_n:=u_{\lambda_n}^-,
\qquad
t_n:=\|u_n\|,
\qquad
v_n:=\frac{u_n}{t_n}.
\]
Then $\|v_n\|=1$ and
$H_{\lambda_n}(u_n)=M(u_n)<0$.
By part \textup{(i)} and the Nehari energy identity,
\[
E_{\lambda_n}(u_n)
=
-\frac{1-\alpha}{2(1+\alpha)}
H_{\lambda_n}(u_n)
=
-\frac{1-\alpha}{2(1+\alpha)}
M(u_n),
\]
we have
\[
H_{\lambda_n}(u_n)\longrightarrow0,
\qquad
M(u_n)\longrightarrow0.
\]
We first prove that $t_n\to0$. Suppose otherwise. Up to a subsequence,
$t_n\to a\in(0,+\infty]$, so $t_n$ is bounded away from zero.
Since $(v_n)$ is bounded in $H_0^1(\Omega)$, up to a subsequence,
\[
v_n\rightharpoonup\bar v
\quad\text{in }H_0^1(\Omega),
\qquad
v_n\to\bar v
\quad\text{in }L^q(\Omega),\quad 1\leq q<2^*.
\]
Since
\[
M(v_n)
=
\frac{M(u_n)}{t_n^{\alpha+1}}
\rightarrow0,
\]
we obtain $M(\bar v)=0$. Since $m\leq0$ and
$m<0$ a.e. in $\Omega\setminus\overline{\Omega^0}$,
$\bar v=0$ a.e. outside $\overline{\Omega^0}$, and therefore
\[
\bar v\in H_0^1(\Omega^0).
\]	
Furthermore,
\[
H_{\lambda_n}(v_n)
=
\frac{H_{\lambda_n}(u_n)}{t_n^2}
\rightarrow0.
\]
Since $\|v_n\|=1$,
\[
1-\lambda_n\|v_n\|_2^2\rightarrow0,
\]
and hence
\[
\|\bar v\|_2^2
=
\frac1{\lambda_{1,\Omega^0}}>0.
\]
The variational characterization of $\lambda_{1,\Omega^0}$ and weak lower
semicontinuity give
\[
1
\leq
\|\bar v\|^2
\leq
\liminf_{n\to\infty}\|v_n\|^2
=
1.
\]
Thus $\bar v$ is the normalized positive first eigenfunction of $\Omega^0$.
Consequently,
\[
\bar v=\widetilde{\varphi}_{1,\Omega^0},
\qquad
v_n\rightarrow\bar v
\quad\text{strongly in }H_0^1(\Omega).
\]	
Dividing \eqref{Eq1} by $t_n$ gives
\begin{equation}\label{eq:EqHM}
	-\Delta v_n-\lambda_n v_n
	=
	t_n^{\alpha-1}m v_n^\alpha
	\qquad\text{in }\Omega.
\end{equation}	
If $a\in(0,+\infty)$, then
$u_n\to a\widetilde{\varphi}_{1,\Omega^0}$ strongly in $H_0^1(\Omega)$,
and
\[
m u_n^\alpha\rightarrow0
\qquad
\text{in }L^{(\alpha+1)/\alpha}(\Omega).
\]
Passing to the limit in the equation for $u_n$ would therefore give
\[
-\Delta\widetilde{\varphi}_{1,\Omega^0}
=
\lambda_{1,\Omega^0}
\widetilde{\varphi}_{1,\Omega^0}
\qquad\text{weakly in }\Omega.
\]
If $a=+\infty$, then $t_n^{\alpha-1}\to0$, and passing to the limit in
\eqref{eq:EqHM} gives the same conclusion.	
Thus, in either case,
$\widetilde{\varphi}_{1,\Omega^0}$ would be a non-negative nontrivial weak
solution of the eigenvalue equation in the whole connected domain $\Omega$.
The strong maximum principle would then imply
$\widetilde{\varphi}_{1,\Omega^0}>0$ in $\Omega$, contradicting its
vanishing in $\Omega\setminus\overline{\Omega^0}$.
Therefore
\[
t_n=\|u_n\|\rightarrow0.
\]	
Finally, the preceding compactness argument applies to every sequence
$\lambda_n\uparrow\lambda_{1,\Omega^0}$. Since the normalized positive
first eigenfunction is unique, every such sequence satisfies
\[
\frac{u_{\lambda_n}^-}{\|u_{\lambda_n}^-\|}
\rightarrow
\widetilde{\varphi}_{1,\Omega^0}
\qquad\text{strongly in }H_0^1(\Omega).
\]
Hence, as $\lambda\uparrow\lambda_{1,\Omega^0}$,
\[
\frac{u_\lambda^-}{\|u_\lambda^-\|}
\rightarrow
\widetilde{\varphi}_{1,\Omega^0}
\qquad\text{strongly in }H_0^1(\Omega).
\]	
\end{proof}
The following estimate yields, in particular,
$\|u_\lambda^-\|_{L^\infty(\Omega)}\to0$ as
$\lambda\uparrow\lambda_{1,\Omega^0}$ in any dimension.

\begin{prop}
\label{prop:linftyEstimate}
Assume that $0<\alpha<1$, $m\in L^\infty(\Omega)$, and
$m\leq0$ a.e. in $\Omega$. Let $u_\lambda\in H_0^1(\Omega)$ be a
non-negative weak solution of \eqref{Eq1}. If $\lambda\leq\Lambda$, then
there exists $C=C(\Omega,N,\Lambda)>0$, independent of $\lambda$, such that
\begin{equation}\label{Estimate Linfinity}
	\|u_\lambda\|_{L^\infty(\Omega)}
	\leq C\|u_\lambda\|.
\end{equation}
\end{prop}

\proof
Since $u_\lambda\geq0$, $m\leq0$, and $\lambda\leq\Lambda$,
\[
-\Delta u_\lambda
=
\lambda u_\lambda+m(x)u_\lambda^\alpha
\leq\Lambda_+u_\lambda,
\qquad
\Lambda_+:=\max\{\Lambda,0\}.
\]
Thus $u_\lambda$ is a non-negative weak subsolution of
$-\Delta u=\Lambda_+u$. The standard $L^\infty$ estimate gives
\[
\|u_\lambda\|_{L^\infty(\Omega)}
\leq C_0\|u_\lambda\|_{L^2(\Omega)}
\leq C_0C_P\|u_\lambda\|,
\]
where the last inequality follows from Poincar\'e's inequality.
This proves \eqref{Estimate Linfinity}.$\fin$
\par
\medskip

Taking $\Lambda=\lambda_{1,\Omega^0}$, the constant in
\eqref{Estimate Linfinity} is uniform for
$\lambda<\lambda_{1,\Omega^0}$. Hence Theorem~\ref{thm:Main2}\textup{(iii)}
immediately gives:

\begin{coro}
\label{Coro L-infinity}
Under the assumptions of Theorem~\ref{thm:Main2}, for
$\lambda\in(\lambda_{1,\Omega},\lambda_{1,\Omega^0})$,
\[
\|u_\lambda^-\|_{L^\infty(\Omega)}
\rightarrow0
\qquad\text{as }\lambda\uparrow\lambda_{1,\Omega^0}.
\]
\end{coro}

\proof
By Proposition~\ref{prop:linftyEstimate},
\[
\|u_\lambda^-\|_{L^\infty(\Omega)}
\leq C\|u_\lambda^-\|,
\]
with $C$ independent of $\lambda<\lambda_{1,\Omega^0}$.
The conclusion follows from Theorem~\ref{thm:Main2} $iii).\fin$

\begin{rem}
The preceding $L^\infty$ convergence will be crucial for proving compactness
of $\operatorname{supp}u_\lambda^-$ when
$\lambda\in[\lambda_{1,\Omega^0}-\varepsilon,\lambda_{1,\Omega^0}]$;
see Corollary~\ref{Coro compact supp}.
\end{rem}

\noindent{\sc Proof of Theorem~\ref{thm1}.}

\medskip
\noindent
\textit{Part $i)$.}
If $|\Omega^+|>0$, then $\Theta^+\neq\varnothing$. For every
$\lambda<\lambda^*$, Lemma~\ref{ExisHomMin}(i) yields a minimizer
$v_\lambda^+\in\Theta^+\cap S$ of
\[
\widehat J_\lambda^+
=
\inf_{v\in\Theta^+\cap S}J_\lambda^+(v).
\]
Set
\[
u_\lambda^+:=t_\lambda(v_\lambda^+)v_\lambda^+,
\qquad
t_\lambda(v)
=
\left(\frac{M(v)}{H_\lambda(v)}\right)^{1/(1-\alpha)}.
\]
By Proposition~\ref{prop:crJE2} $i)$, $u_\lambda^+$ is a non-negative weak
solution of \eqref{Eq1}, with
\[
E_\lambda(u_\lambda^+)
=
\widehat E_\lambda^+
=
\widehat J_\lambda^+<0,
\qquad
E_\lambda''(u_\lambda^+)>0.
\]
Lemma~\ref{Lemm3.3} gives the required regularity.
Continuity and strict monotonicity of
$\lambda\mapsto E_\lambda(u_\lambda^+)$ follow from
Corollary~\ref{cor:monot}, while Lemma~\ref{lem:corD}(b) yields
\[
E_\lambda(u_\lambda^+)\longrightarrow0
\qquad\text{as }\lambda\to-\infty.
\]
This proves $i)$.

\medskip
\noindent
\textit{Part $ii)$.}
The bifurcation at the infinity is consequence from \cite{Dias-Hern} and \cite{Hernandez tesis}. Assume
\[
\int_\Omega m(x)\varphi_{1,\Omega}^{\alpha+1}\,dx<0,
\qquad
|\Omega^+\cup\Omega^0|>0.
\]
Then
\[
0<\lambda_{1,\Omega}<\lambda^*<+\infty.
\]
For every
$\lambda\in(\lambda_{1,\Omega},\lambda^*)$,
Lemma~\ref{ExisHomMin}(ii) gives a minimizer
$v_\lambda^-\in\Theta_\lambda^-\cap S$. Setting
\[
u_\lambda^-:=t_\lambda(v_\lambda^-)v_\lambda^-,
\]
Proposition~\ref{prop:crJE2}(ii) gives a non-negative weak solution with
\[
E_\lambda(u_\lambda^-)
=
\widehat E_\lambda^-
=
\widehat J_\lambda^->0,
\qquad
E_\lambda''(u_\lambda^-)<0.
\]
Regularity follows from Lemma~\ref{Lemm3.3}.
Corollary~\ref{cor:monot} yields continuity and strict monotonicity of
\[
(\lambda_{1,\Omega},\lambda^*)\ni\lambda
\longmapsto E_\lambda(u_\lambda^-).
\]
Moreover, Lemma~\ref{lem:corD}(a) gives
\[
E_\lambda(u_\lambda^-)\longrightarrow+\infty
\qquad\text{as }\lambda\downarrow\lambda_{1,\Omega},
\]
while, under condition \textbf{(M)}, Lemma~\ref{solextr}(b) gives
\[
E_\lambda(u_\lambda^-)\longrightarrow0
\qquad\text{as }\lambda\uparrow\lambda^*.
\]
It remains to justify the endpoint limit in the degenerate case without
condition \textbf{(M)}. Assume $m\leq0$ a.e. in $\Omega$ and that
$\Omega^0=\operatorname{int}\{m=0\}$ is a nonempty connected Lipschitz
domain compactly contained in $\Omega$.Assume also that
\[
m<0
\qquad\text{a.e. in }
\Omega\setminus\overline{\Omega^0}.
\]
By Theorem~\ref{thm:ExisExtr}(2),
\begin{equation}
\lambda^*=\lambda_{1,\Omega^0}.
\label{landaEstrella}
\end{equation}
Let $\phi:=\widetilde{\varphi}_{1,\Omega^0}$ be the zero extension of a
first eigenfunction of $\Omega^0$. Then
\begin{equation}
H_{\lambda^*}(\phi)=0,
\qquad
M(\phi)=0.
\label{Hfi}
\end{equation}
We claim that there exists $\zeta\in H_0^1(\Omega)$ such that
\begin{equation}
B_*(\phi,\zeta)<0,
\qquad
M(\zeta)<0,
\label{BEstrella}
\end{equation}
where
\[
B_*(v,w)
:=
\int_\Omega\nabla v\cdot\nabla w\,dx
-\lambda^*\int_\Omega vw\,dx.
\]
Indeed, the functional
$\zeta\mapsto B_*(\phi,\zeta)$ cannot vanish identically; otherwise
$\phi$ would satisfy
\[
-\Delta\phi=\lambda^*\phi
\]
weakly in all of $\Omega$, contradicting the strong maximum principle,
since $\phi$ vanishes identically in
$\Omega\setminus\overline{\Omega^0}$.
Thus there exists $\zeta_0\in H_0^1(\Omega)$ such that $B_*(\phi,\zeta_0)<0$. Choose a nonzero function $\eta\in C_c^\infty
\bigl(\Omega\setminus\overline{\Omega^0}\bigr).$ Since $m<0$ a.e. in
$\Omega\setminus\overline{\Omega^0}$, $M(\eta)<0,$ whereas, because $\phi=0$ a.e. on the support of $\eta$, $B_*(\phi,\eta)=0$.
Set $\zeta:=\zeta_0+s\eta$. For $s>0$ sufficiently large,
\[
M(\zeta)<0,
\qquad
B_*(\phi,\zeta)
=
B_*(\phi,\zeta_0)<0.
\]
This proves \eqref{BEstrella}.
Set
\[
\delta:=\lambda^*-\lambda>0,
\qquad
\varepsilon_\lambda:=K\delta,
\qquad
w_\lambda:=\phi+\varepsilon_\lambda\zeta.
\]
Since $m=0$ on $\Omega^0$ and $\phi=0$ outside $\Omega^0$,
\begin{equation}
M(w_\lambda)
=
K^{\alpha+1}\delta^{\alpha+1}M(\zeta)<0.
\label{MK}
\end{equation}
Moreover,
\[
H_\lambda(w_\lambda)
=
2\varepsilon_\lambda B_*(\phi,\zeta)
+\varepsilon_\lambda^2H_{\lambda^*}(\zeta)
+\delta\|\phi+\varepsilon_\lambda\zeta\|_2^2,
\]
and hence
\begin{equation}
H_\lambda(w_\lambda)
=
\delta\bigl[\|\phi\|_2^2+2KB_*(\phi,\zeta)\bigr]
+O(\delta^2).
\label{eq:H-wlambda-expansion}
\end{equation}
Choosing $K$ sufficiently large gives
\begin{equation}
H_\lambda(w_\lambda)\leq-c\delta<0
\qquad(0<\delta<\delta_0)
\label{Hdelta}
\end{equation}
for suitable $c,\delta_0>0$.
Thus $w_\lambda\in\Theta_\lambda^-$ for $\lambda$ close to $\lambda^*$.
By minimality and the homogeneity of $J_\lambda^-$,
\[
0<E_\lambda(u_\lambda^-)
\leq J_\lambda^-(w_\lambda).
\]
Using \eqref{j22}, \eqref{MK}, and \eqref{Hdelta},
\[
J_\lambda^-(w_\lambda)
\leq
C\delta^{(1+\alpha)/(1-\alpha)}.
\]
Therefore
\begin{equation}
E_\lambda(u_\lambda^-)\longrightarrow0
\qquad\text{as }
\lambda\uparrow\lambda^*
=
\lambda_{1,\Omega^0}.
\label{Egoes0}
\end{equation}
Thus the degenerate endpoint relation holds without condition
\textbf{(M)}. This completes the proof of Theorem~\ref{thm1}.$\fin$

\section{Compact support solutions}

As pointed out in the Introduction, our main goal is to identify, among the
non-negative solutions obtained by the Nehari manifold method, flat
solutions satisfying $u>0$ in $\Omega$ and
$\partial u/\partial n=0$ on $\partial\Omega$, as well as solutions
$u\geq0$ with $\operatorname{supp}u\Subset\Omega$.

In the one-dimensional case $N=1$, classical energy methods provide a
complete description of the solution set. For \eqref{Eq1} with
$\Omega=(-1,1)$ and $m\equiv-1$, the branch of non-negative solutions
bifurcating from $\lambda_1$ passes from positive solutions with
$u'(\pm1)\neq0$ to compactly supported solutions. All these solutions are
unstable. At a unique intermediate value $\lambda^*$, the branch contains
a flat positive solution satisfying $u'(\pm1)=0$; see
\cite{DH-CRAS,DH Portugalia}. Interestingly, for $m\equiv-1$ and $N\geq3$,
stable compactly supported solutions may also occur; see \cite{DHI China}.

For $N>1$, only partial results seem to be available. In
\cite{DHIlyaNonlinear}, for $m\leq0$ and star-shaped domains, we derived
through a Pohozaev-type identity a necessary condition for solutions
satisfying $\partial u/\partial n=0$ on $\partial\Omega$; see also
\cite{DHI China}. We now obtain sufficient conditions for compact support,
using the local supersolution method of \cite{Diaz Pitman}. The following
result extends \cite[Theorem~5.1]{DHIlyaNonlinear}.

For
$q_0\in(0,\|m^-\|_{L^\infty(\Omega)}]$, where
$m^-(x):=\max\{-m(x),0\}$, set
\[
\Omega_{q_0}^-:=\{x\in\Omega:m^-(x)\geq q_0\}.
\]
For a non-negative solution $u_\lambda$ of \eqref{Eq1} and $M>0$, we also
write
\[
[0\leq u_\lambda\leq M]
:=
\{x\in\Omega:0\leq u_\lambda(x)\leq M\}.
\]

Since $0<\alpha<1$, the function
$f_{\lambda,q_0}(s):=q_0s^\alpha-\lambda s$
is non-decreasing on $[0,\delta_{\lambda,q_0}]$, where
\begin{equation}\label{eq:delta}
\delta_{\lambda,q_0}
:=
\left(\frac{\alpha q_0}{\lambda}\right)^{\frac1{1-\alpha}}
\quad\text{if }\lambda>0,
\qquad
\delta_{\lambda,q_0}:=+\infty
\quad\text{if }\lambda\leq0.
\end{equation}
Indeed,
$f_{\lambda,q_0}'(s)=\alpha q_0s^{\alpha-1}-\lambda\geq0$
for $0<s\leq\delta_{\lambda,q_0}$.

For $\mu>0$, define
\[
\psi_{\mu,q_0}(\tau)
:=
\frac1{\sqrt{2\mu}}
\int_0^\tau
\frac{ds}{
\sqrt{\displaystyle
	\frac{q_0}{\alpha+1}s^{\alpha+1}-\frac{\lambda}{2}s^2}},
\qquad
0\leq\tau<\delta_{\lambda,q_0}.
\]
The integrand is positive on $(0,\delta_{\lambda,q_0})$, and its
singularity at $0$ is integrable. Hence $\psi_{\mu,q_0}$ is well defined,
continuous, and strictly increasing.

We shall use the following local dead-core criterion.

\begin{theo}
\label{thm:Comp}
Let $u_\lambda\in H_0^1(\Omega)\cap C(\overline\Omega)$ be a non-negative
weak solution of \eqref{Eq1}, and let $G\subset\Omega$ be open. Assume
\begin{equation}\label{eq:G-weight}
	-m(x)\geq q_0
	\qquad\text{for a.e. }x\in G
\end{equation}
and
\begin{equation}\label{eq:G-smallness}
	0\leq u_\lambda(x)\leq\delta_{\lambda,q_0}
	\qquad\text{for every }x\in G.
\end{equation}
If $\lambda>0$, assume additionally that
$\|u_\lambda\|_{L^\infty(\Omega)}<\delta_{\lambda,q_0}$.

For $x_0\in G$, set
\[
d_G(x_0)
:=
\operatorname{dist}(x_0,\partial G\setminus\partial\Omega),
\]
with $d_G(x_0)=+\infty$ if
$\partial G\setminus\partial\Omega=\emptyset$.
If
\begin{equation}\label{eq:Hsfb}
	d_G(x_0)
	\geq
	\psi_{1/N,q_0}
	\bigl(\|u_\lambda\|_{L^\infty(\Omega)}\bigr),
\end{equation}
then $u_\lambda(x_0)=0$. Consequently,
\[
\{u_\lambda>0\}\cap G
\subset
\left\{
x\in G:
d_G(x)<
\psi_{1/N,q_0}
\bigl(\|u_\lambda\|_{L^\infty(\Omega)}\bigr)
\right\},
\]
and therefore
\[
\operatorname{supp}u_\lambda\cap G
\subset
\left\{
x\in G:
d_G(x)\leq
\psi_{1/N,q_0}
\bigl(\|u_\lambda\|_{L^\infty(\Omega)}\bigr)
\right\}.
\]
\end{theo}

\proof
Set
\[
R:=
\psi_{1/N,q_0}
\bigl(\|u_\lambda\|_{L^\infty(\Omega)}\bigr).
\]
If $R=0$, the conclusion is immediate. Assume $R>0$. By
\eqref{eq:Hsfb},
$D_R(x_0):=B_R(x_0)\cap\Omega\subset G$.
Let $\eta_{1/N,q_0}:=\psi_{1/N,q_0}^{-1}$ and define
\[
U(x):=\eta_{1/N,q_0}(|x-x_0|),
\qquad x\in D_R(x_0).
\]
Then $U(x_0)=0$,
$0\leq U\leq\|u_\lambda\|_{L^\infty(\Omega)}$ in $D_R(x_0)$, and
$U=\|u_\lambda\|_{L^\infty(\Omega)}$ on
$\partial B_R(x_0)\cap\Omega$. By the local barrier construction in
\cite[Theorem~1.5]{Diaz Pitman},
\begin{equation}\label{eq:barrier-super}
	-\Delta U+q_0U^\alpha-\lambda U\geq0
	\qquad\text{in }D_R(x_0).
\end{equation}	
On the other hand, \eqref{Eq1} and \eqref{eq:G-weight} give
\[
-\Delta u_\lambda+q_0u_\lambda^\alpha-\lambda u_\lambda
=
(m(x)+q_0)u_\lambda^\alpha\leq0
\qquad\text{in }D_R(x_0).
\]
Moreover, $u_\lambda\leq U$ on $\partial D_R(x_0)$: this follows from
the definition of $U$ on the spherical part and from the homogeneous
Dirichlet condition on $\partial\Omega$.

Since $s\mapsto q_0s^\alpha-\lambda s$ is non-decreasing on
$[0,\delta_{\lambda,q_0}]$, the comparison principle gives
$0\leq u_\lambda\leq U$ in $D_R(x_0)$. Hence
$0\leq u_\lambda(x_0)\leq U(x_0)=0$, proving the assertion.$\fin$
\par
\medskip
We can now prove Theorem~\ref{thm3}.
\par
\noindent
{\sc Proof of Theorem~\ref{thm3}}
Choose $q_0\in(0,m_{\rho_0}^-)$.
By \eqref{eq:negative-boundary-weight},
$-m(x)\geq q_0$ a.e. in $\mathcal C_{\rho_0}$. Set
\[
\rho:=\frac{\rho_0}{4},
\qquad
G:=\mathcal C_{\rho_0/2}
=
\left\{
x\in\Omega:
\operatorname{dist}(x,\partial\Omega)<\frac{\rho_0}{2}
\right\}.
\]
Then $G\subset\mathcal C_{\rho_0}\subset\Omega_{q_0}^-$ and, for every
$x_0\in\mathcal C_\rho$,
\begin{equation}\label{eq:collar-distance}
	\operatorname{dist}
	(x_0,\partial G\setminus\partial\Omega)
	\geq\frac{\rho_0}{4}.
\end{equation}	
By Theorem~\ref{thm:Main2}\textup{(iii)} and
Corollary~\ref{Coro L-infinity},
\[
\|u_\lambda^-\|_{L^\infty(\Omega)}
\longrightarrow0
\qquad
\text{as }\lambda\uparrow\lambda_{1,\Omega^0}.
\]
Hence, for $\lambda$ sufficiently close to $\lambda_{1,\Omega^0}$,
$\|u_\lambda^-\|_\infty<\delta_{\lambda,q_0}$ and
\begin{equation}\label{eq:psi-small}
	\psi_{1/N,q_0}
	\bigl(\|u_\lambda^-\|_{L^\infty(\Omega)}\bigr)
	<
	\frac{\rho_0}{4}.
\end{equation}
Thus there exists
\[
0<\varepsilon<
\lambda_{1,\Omega^0}-\lambda_{1,\Omega}
\]
such that these inequalities hold whenever
$\lambda\in
(\lambda_{1,\Omega^0}-\varepsilon,\lambda_{1,\Omega^0})$.	
For such $\lambda$ and $x_0\in\mathcal C_\rho$,
\eqref{eq:collar-distance} and~\eqref{eq:psi-small} allow us to apply
Theorem~\ref{thm:Comp}, yielding $u_\lambda^-(x_0)=0$.
Therefore
\[
u_\lambda^-=0\quad\text{in }\mathcal C_\rho,
\qquad
\operatorname{supp}u_\lambda^-
\subset
\{x\in\Omega:\operatorname{dist}(x,\partial\Omega)\geq\rho\}
\Subset\Omega.
\]
\fineq

As announced in Remark~13 of
\cite{DiazHernandezIlyasov2026}, we next give a useful
$L^\infty$ estimate.

\begin{lemma}
\label{lem:Linfty-negative-lambda}
Assume that $0<\alpha<1$, $m\in L^\infty(\Omega)$, and $\lambda<0$.
If $u_\lambda$ is a non-negative weak solution of \eqref{Eq1}, then
\begin{equation}\label{eq:Linfty-negative-lambda}
	\|u_\lambda\|_{L^\infty(\Omega)}
	\leq
	\left(
	\frac{\|m^+\|_{L^\infty(\Omega)}}{|\lambda|}
	\right)^{\frac1{1-\alpha}}.
\end{equation}
In particular,
$\|u_\lambda\|_{L^\infty(\Omega)}\to0$ as $\lambda\to-\infty$.
\end{lemma}

\proof
Set
\[
K_\lambda
:=
\left(
\frac{\|m^+\|_{L^\infty(\Omega)}}{|\lambda|}
\right)^{\frac1{1-\alpha}}
\]
and let $w:=(u_\lambda-K_\lambda)^+\in H_0^1(\Omega)$.
Testing \eqref{Eq1} by $w$ gives
\[
\int_{\{u_\lambda>K_\lambda\}}|\nabla w|^2\,dx
\leq
\int_{\{u_\lambda>K_\lambda\}}
\bigl(
-|\lambda|u_\lambda
+\|m^+\|_\infty u_\lambda^\alpha
\bigr)w\,dx.
\]
On $\{u_\lambda>K_\lambda\}$,
\[
u_\lambda^{1-\alpha}
>
K_\lambda^{1-\alpha}
=
\frac{\|m^+\|_\infty}{|\lambda|},
\]
so the integrand on the right is negative. Hence
$\int|\nabla w|^2\leq0$, and therefore $w=0$.
Thus $u_\lambda\leq K_\lambda$ a.e. in $\Omega$, proving
\eqref{eq:Linfty-negative-lambda}.$\fin$

We can now complete the proof of Theorem~\ref{thm4}.
\par
\noindent
{\sc Proof of Theorem~\ref{thm4}}
Choose $q_0\in(0,m_\partial^-)$.
By \eqref{eq:negative-boundary-collar-thm4} and
\eqref{eq:uniform-negative-boundary-thm4},
$-m(x)\geq q_0$ a.e. in $\mathcal C_{\rho_0}$.	Set
\[
\rho:=\frac{\rho_0}{4},
\qquad
G:=\mathcal C_{\rho_0/2}.
\]
Then $G\subset\mathcal C_{\rho_0}\subset\Omega_{q_0}^-$ and, for every
$x_0\in\mathcal C_\rho$,
\begin{equation}\label{eq:distance-thm4}
	\operatorname{dist}
	(x_0,\partial G\setminus\partial\Omega)
	\geq\frac{\rho_0}{4}.
\end{equation}	
By Lemma~\ref{lem:Linfty-negative-lambda},
$\|u_\lambda^+\|_{L^\infty(\Omega)}\to0$ as $\lambda\to-\infty$.
Therefore
\[
\psi_{1/N,q_0}
\bigl(\|u_\lambda^+\|_{L^\infty(\Omega)}\bigr)
\longrightarrow0.
\]
Hence there exists $\sigma<0$ such that, for every $\lambda<\sigma$,
\begin{equation}\label{eq:psi-small-thm4}
	\psi_{1/N,q_0}
	\bigl(\|u_\lambda^+\|_{L^\infty(\Omega)}\bigr)
	<
	\frac{\rho_0}{4}.
\end{equation}
Combining \eqref{eq:distance-thm4} and \eqref{eq:psi-small-thm4},
Theorem~\ref{thm:Comp} gives $u_\lambda^+=0$ in $\mathcal C_\rho$.
Consequently,
\[
\operatorname{supp}u_\lambda^+
\subset
\{x\in\Omega:\operatorname{dist}(x,\partial\Omega)\geq\rho\}
\Subset\Omega,
\]
so $u_\lambda^+$ has compact support in $\Omega.\fin$

\begin{coro}
\label{Coro compact supp}
Assume that $0<\alpha<1$, $m\leq0$ a.e. in $\Omega$,
$|\Omega^0|>0$, and $m\in L^\infty(\Omega)$. Suppose that there exist
$\rho_0>0$ and $m_{\rho_0}^-$ such that
\[
\mathcal C_{\rho_0}
:=
\{x\in\Omega:\operatorname{dist}(x,\partial\Omega)<\rho_0\}
\subset\Omega^-,
\qquad
-m(x)\geq m_{\rho_0}^-
\quad\text{a.e. in }\mathcal C_{\rho_0},
\]
and assume that
$\lambda_{1,\Omega}<\lambda_{1,\Omega^0}$.
Then there exist $\varepsilon>0$ and $\rho\in(0,\rho_0)$ such that, for
every
$\lambda\in
(\lambda_{1,\Omega^0}-\varepsilon,\lambda_{1,\Omega^0})$,
\[
u_\lambda^-=0
\quad\text{in }
\{x\in\Omega:\operatorname{dist}(x,\partial\Omega)<\rho\}.
\]
In particular,
\[
\operatorname{supp}u_\lambda^-
\subset
\{x\in\Omega:\operatorname{dist}(x,\partial\Omega)\geq\rho\}
\Subset\Omega.
\]
\end{coro}

\proof
This is an immediate consequence of Theorem~\ref{thm3}.$\fin$
\par
\medskip
The following result concerns sign-changing weights. In this case,
compact support requires, in addition, a negative boundary collar and a
suitable smallness condition on the solution.

\begin{coro}
\label{cor:sign-changing-compact-support}
Assume that $0<\alpha<1$, $m\in L^\infty(\Omega)$, $|\Omega^0|>0$, and
\[
\int_\Omega m(x)\varphi_{1,\Omega}^{\alpha+1}\,dx<0.
\]
Suppose that there exist $\rho_0>0$ and $m_{\rho_0}^->0$ such that
\[
\mathcal C_{\rho_0}
:=
\{x\in\Omega:\operatorname{dist}(x,\partial\Omega)<\rho_0\}
\subset\{m<0\},
\qquad
-m(x)\geq m_{\rho_0}^-
\quad\text{a.e. in }\mathcal C_{\rho_0}.
\]
Assume, in addition, that for some
$\lambda_c\in(\lambda_{1,\Omega},\lambda_{1,\Omega^0})$,
equation \eqref{Eq1} has a non-negative weak solution $u_{\lambda_c}^-$
satisfying $E_{\lambda_c}''(u_{\lambda_c}^-)<0$ and, for some
$q_0\in(0,m_\rho^-)$,
\[
\psi_{1/N,q_0}
\bigl(\|u_{\lambda_c}^-\|_{L^\infty(\Omega)}\bigr)
<
\frac{\rho_0}{4},
\qquad
\|u_{\lambda_c}^-\|_{L^\infty(\Omega)}
<
\delta_{\lambda_c,q_0}.
\]
Then $u_{\lambda_c}^-$ has compact support in $\Omega$.
\end{coro}

\proof
Set
\[
G:=\mathcal C_{\rho_0/2},
\qquad
\rho:=\frac{\rho_0}{4}.
\]
Since $q_0<m_{\rho_0}^-$, we have $-m(x)\geq q_0$ a.e. in $G$, and for every
$x_0\in\mathcal C_\rho$,
\[
\operatorname{dist}
(x_0,\partial G\setminus\partial\Omega)
\geq\frac{\rho_0}{4}.
\]
The remaining assumptions of Theorem~\ref{thm:Comp} follow from the
hypotheses. Hence $u_{\lambda_c}^-=0$ in $\mathcal C_\rho$, and therefore
\[
\operatorname{supp}u_{\lambda_c}^-
\subset
\{x\in\Omega:\operatorname{dist}(x,\partial\Omega)\geq\rho\}
\Subset\Omega.
\]
\fineq

\subsection{A one-dimensional threshold for the formation of compact support}

The preceding results suggest a qualitative connection between the
bifurcation diagram for the strictly negative weight $m\equiv-1$ and the
diagrams corresponding to non-positive weights with a nontrivial zero
region. This connection is especially transparent in the one-dimensional
setting. Let $\Omega=(-1,1)$ and for $n\ge2$, consider
\[
m_n(x)=
\begin{cases}
0,& |x|<1/n,\\[1mm]
-1,& 1/n\le |x|<1.
\end{cases}
\]
We denote
\[
\Omega_n^0:=\left(-\frac1n,\frac1n\right).
\]
For later use, we also set
\[
M_n(u):=\int_{-1}^1m_n(x)|u|^{\alpha+1}\,dx.
\]
Then
\[
\lambda_n^*
=
\lambda_{1,\Omega_n^0}
=
\left(\frac{n\pi}{2}\right)^2.
\]

\begin{theo}[Explicit threshold, compact-support branch and onset amplitude]\label{thm:explicit-threshold-1d}
Let $0<\alpha<1$ and $n\ge2$. Consider the problem
\[
\begin{cases}
	-u''=\lambda u+m_n(x)u^\alpha,& -1<x<1,\\
	u(-1)=u(1)=0.
\end{cases}
\]
Define
\[
\lambda_{c,n}:=
\frac{\pi^2}
{\left((1-\alpha)+\dfrac{1+\alpha}{n}\right)^2}.
\]
Then
\[
\frac{\pi^2}{4}
<
\lambda_{c,n}
<
\lambda_n^*.
\]
At $\lambda=\lambda_{c,n}$ there exists an even positive flat solution
$u_{c,n}$ satisfying $u_{c,n}(\pm1)=u_{c,n}'(\pm1)=0$. For every
$\lambda\in(\lambda_{c,n},\lambda_n^*)$ there exists a unique even
non-negative solution in the above class whose connected support is
\[
\operatorname{supp}u_{\lambda,n}
=
[-b_{\lambda,n},b_{\lambda,n}]
\Subset(-1,1),
\]
where
\[
b_{\lambda,n}
=
\frac{1}{1-\alpha}
\left(
\frac{\pi}{\sqrt{\lambda}}
-\frac{1+\alpha}{n}
\right).
\]
For $x\ge0$ the solution is
\[
u_{\lambda,n}(x)
=
\begin{cases}
	A_{\lambda,n}\cos(\sqrt{\lambda}\,x),
	& 0\le x\le\dfrac1n,\\[2mm]
	\displaystyle
	\left(
	\frac{2}{(1+\alpha)\lambda}
	\right)^{\frac1{1-\alpha}}
	\sin^{\frac{2}{1-\alpha}}
	\left[
	\frac{1-\alpha}{2}\sqrt{\lambda}\,
	(b_{\lambda,n}-x)
	\right],
	& \dfrac1n\le x\le b_{\lambda,n},\\[3mm]
	0,
	& b_{\lambda,n}\le x\le1,
\end{cases}
\]
and it is extended evenly to $(-1,0)$. Moreover,
\[
A_{\lambda,n}
=
\|u_{\lambda,n}\|_{L^\infty(-1,1)}
=
\left(
\frac{2}{(1+\alpha)\lambda}
\right)^{\frac1{1-\alpha}}
\cos^{\frac{1+\alpha}{1-\alpha}}
\left(\frac{\sqrt{\lambda}}{n}\right).
\]
For each fixed $n$, the mapping $\lambda\longmapsto A_{\lambda,n}$
is strictly decreasing on $(\lambda_{c,n},\lambda_n^*)$, and
$A_{\lambda,n}\longrightarrow0\hbox{ as }\lambda\uparrow\lambda_n^*$. If
\[
A_n:=\|u_{c,n}\|_{L^\infty(-1,1)},
\]
then
\[
A_n=
\left[
\frac{2}{(1+\alpha)\pi^2}
\left(
(1-\alpha)+\frac{1+\alpha}{n}
\right)^2
\right]^{\frac1{1-\alpha}}
\cos^{\frac{1+\alpha}{1-\alpha}}
\left(
\frac{\pi}{n(1-\alpha)+(1+\alpha)}
\right),
\]
and
\[
\lambda_{c,n}\uparrow
\lambda_{c,\infty}:=
\frac{\pi^2}{(1-\alpha)^2},
\]
while
\[
A_n\longrightarrow
A_\infty:=
\left[
\frac{2(1-\alpha)^2}
{(1+\alpha)\pi^2}
\right]^{\frac1{1-\alpha}}
>0.
\]
Furthermore,
\[
\frac{\lambda_{c,n}}{\lambda_n^*}\longrightarrow0.
\]
Moreover, $H_\lambda(u_{\lambda,n})=M_n(u_{\lambda,n})<0$, and therefore
$E_\lambda(u_{\lambda,n})>0,\ E_\lambda''(u_{\lambda,n})<0$. Thus, these
explicit solutions have the variational character of the
branch $u_{\lambda,n}^-$.
\end{theo}

\proof
Set $a_n:=\frac1n$. We work on $[0,1]$ and use even reflection.
Assume that the positivity set is $[0,b)$, with $a_n<b\le1,$ and that the
solution is flat at the free boundary: $u(b)=u'(b)=0.$ On $(a_n,b)$ one
has $m_n=-1$, and therefore
\[
-u''+u^\alpha=\lambda u.
\]
Multiplying by $u'$ and integrating from $x$ to $b$ gives
\[
\frac12(u')^2
=
\frac{1}{1+\alpha}u^{1+\alpha}
-\frac{\lambda}{2}u^2.
\]
Hence
\[
(u')^2
=
\frac{2}{1+\alpha}u^{1+\alpha}
-\lambda u^2.
\]
Since $u'<0$ in $(a_n,b)$,
\[
-u'
=
u^{(1+\alpha)/2}
\left(
\frac{2}{1+\alpha}
-\lambda u^{1-\alpha}
\right)^{1/2}.
\]
Therefore,
\[
b-x
=
\int_0^{u(x)}
\frac{ds}{
	s^{(1+\alpha)/2}
	\sqrt{\dfrac{2}{1+\alpha}-\lambda s^{1-\alpha}}
}.
\]
With
\[
y=s^{(1-\alpha)/2},
\qquad
s^{-(1+\alpha)/2}\,ds
=
\frac{2}{1-\alpha}\,dy,
\]
we obtain
\[
b-x
=
\frac{2}{(1-\alpha)\sqrt{\lambda}}
\arcsin
\left[
\sqrt{\frac{(1+\alpha)\lambda}{2}}\,
u(x)^{(1-\alpha)/2}
\right].
\]
Equivalently,
\[
u(x)
=
\left(
\frac{2}{(1+\alpha)\lambda}
\right)^{\frac{1}{1-\alpha}}
\sin^{\frac{2}{1-\alpha}}
\left[
\frac{1-\alpha}{2}\sqrt{\lambda}\,(b-x)
\right].
\]
Thus
\[
u(x)=
C_\lambda
\sin^p\!\left[\kappa_\lambda(b-x)\right],
\]
where
\[
C_\lambda=
\left(
\frac{2}{(1+\alpha)\lambda}
\right)^{1/(1-\alpha)},
\qquad
p=\frac{2}{1-\alpha},
\qquad
\kappa_\lambda=
\frac{1-\alpha}{2}\sqrt{\lambda}.
\]
On $(0,a_n)$, where $m_n=0$, $-u''=\lambda u$. The symmetry condition
$u'(0)=0$ yields
\[
u(x)=A\cos(\sqrt{\lambda}\,x).
\]
We now explain carefully the matching at $x=a_n$. Since the
coefficient $m_n$ is discontinuous there, the two formulas above are
obtained independently on the two subintervals. In order that their
union define a global weak solution, no Dirac mass may appear in
$-u''$ at $x=a_n$. Thus both $u$ and $u'$ must be continuous at the
interface
\[
u_{\rm in}(a_n)=u_{\rm out}(a_n),
\qquad
u_{\rm in}'(a_n)=u_{\rm out}'(a_n).
\]
Instead of solving these two equations simultaneously, it is more
convenient to divide the second by the first. This eliminates the
unknown amplitudes and gives equality of the logarithmic derivatives
\[
\frac{u_{\rm in}'(a_n)}{u_{\rm in}(a_n)}
=
\frac{u_{\rm out}'(a_n)}{u_{\rm out}(a_n)}.
\]
For the inner profile,
\[
\frac{u_{\rm in}'(a_n)}{u_{\rm in}(a_n)}
=
-\sqrt{\lambda}\tan(\sqrt{\lambda}\,a_n).
\]
For the outer profile,
\[
\frac{u_{\rm out}'(a_n)}{u_{\rm out}(a_n)}
=
-p\kappa_\lambda
\cot\left[\kappa_\lambda(b-a_n)\right].
\]
Since $p\kappa_\lambda=\sqrt{\lambda}$, the matching condition becomes
\[
\cot\left[\kappa_\lambda(b-a_n)\right]
=
\tan(\sqrt{\lambda}\,a_n).
\]
Along the first positive branch the two arguments lie in the relevant
principal interval, and hence
\[
\kappa_\lambda(b-a_n)
+\sqrt{\lambda}\,a_n
=
\frac{\pi}{2}.
\]
Therefore
\[
b=
\frac{1}{1-\alpha}
\left(
\frac{\pi}{\sqrt{\lambda}}
-\frac{1+\alpha}{n}
\right).
\]
The boundary dead core is born when $b=1$, and this yields
\[
\lambda_{c,n}
=
\frac{\pi^2}
{\left((1-\alpha)+\dfrac{1+\alpha}{n}\right)^2}.
\]
Moreover,
\[
b<1\quad\Longleftrightarrow\quad\lambda>\lambda_{c,n},
\]
while
\[
b>\frac1n\quad\Longleftrightarrow\quad\lambda<\lambda_n^*.
\]
Thus, compact support occurs precisely for
$\lambda\in(\lambda_{c,n},\lambda_n^*)$.	
To compute the amplitude, use again the phase relation
\[
\kappa_\lambda(b-a_n)
=
\frac{\pi}{2}-\frac{\sqrt{\lambda}}{n}.
\]
Hence
\[
\sin\left[\kappa_\lambda(b-a_n)\right]
=
\cos\left(\frac{\sqrt{\lambda}}{n}\right).
\]
The continuity of $u$ at $x=a_n$ gives
\[
A_{\lambda,n}
\cos\left(\frac{\sqrt{\lambda}}{n}\right)
=
C_\lambda
\cos^{2/(1-\alpha)}
\left(\frac{\sqrt{\lambda}}{n}\right).
\]
Since $\sqrt{\lambda}/n<\pi/2$,
\[
A_{\lambda,n}
=
C_\lambda
\cos^{(1+\alpha)/(1-\alpha)}
\left(\frac{\sqrt{\lambda}}{n}\right).
\]
For fixed $n$,
\[
\log A_{\lambda,n}
=
-\frac{1}{1-\alpha}\log\lambda
+
\frac{1+\alpha}{1-\alpha}
\log\cos\left(\frac{\sqrt{\lambda}}{n}\right)
+\hbox{constant},
\]
so
\[
\frac{d}{d\lambda}\log A_{\lambda,n}
=
-\frac{1}{(1-\alpha)\lambda}
-
\frac{1+\alpha}{1-\alpha}
\frac{\tan(\sqrt{\lambda}/n)}
{2n\sqrt{\lambda}}
<0.
\]
Thus $\lambda\mapsto A_{\lambda,n}$ is strictly decreasing. Since
\[
\frac{\sqrt{\lambda}}{n}
\uparrow\frac{\pi}{2}
\quad\hbox{as }\lambda\uparrow\lambda_n^*,
\]
one has $A_{\lambda,n}\to0$. Taking $\lambda=\lambda_{c,n}$ gives the
formula for $A_n$. Since
\[
\lambda_{c,n}\uparrow
\frac{\pi^2}{(1-\alpha)^2}
\]
and
\[
\frac{\sqrt{\lambda_{c,n}}}{n}
=
\frac{\pi}{n(1-\alpha)+(1+\alpha)}
\longrightarrow0,
\]
we obtain
\[
A_n\to
A_\infty=
\left[
\frac{2(1-\alpha)^2}
{(1+\alpha)\pi^2}
\right]^{1/(1-\alpha)}.
\]
Furthermore,
	\[
	\frac{\lambda_{c,n}}{\lambda_n^*}
	=
	\frac{4}
	{\left(n(1-\alpha)+(1+\alpha)\right)^2}
	\longrightarrow0
	\qquad\text{as }n\to\infty.
	\]
Finally, testing the equation by $u_{\lambda,n}$ gives
\[
H_\lambda(u_{\lambda,n})
=
M_n(u_{\lambda,n}).
\]
Since $m_n=-1$ on a set where $u_{\lambda,n}>0$,
$M_n(u_{\lambda,n})<0$. The usual Nehari identities therefore yield
\[
E_\lambda(u_{\lambda,n})>0,
\qquad
E_\lambda''(u_{\lambda,n})<0.
\]
\fineq

\begin{rem}
The limiting values
\[
\lambda_{c,\infty}
=
\frac{\pi^2}{(1-\alpha)^2},
\qquad
A_\infty
=
\left[
\frac{2(1-\alpha)^2}
{(1+\alpha)\pi^2}
\right]^{1/(1-\alpha)}
\]
coincide with the one-dimensional flat-solution constants obtained in
\cite{Diaz ambiguity}, after the identifications
$R=1,\qquad V_0=1,\qquad m=\alpha$. Thus, at the level of the critical parameter and the onset amplitude,
	the limit $n\to\infty$ recovers exactly the flat-solution constants of the
	strictly negative problem $m\equiv-1$.
\end{rem}

It is important to get some additional information about the amplitudes
$A_n$. We have

\begin{coro}[Comparison of the onset amplitudes for $n\geq3$]
\label{cor:An-greater-Ainfty}
Let $0<\alpha<1$, and let $A_n$ denote the $L^\infty$-amplitude of
the flat positive solution at the threshold $\lambda=\lambda_{c,n}$,
as in the preceding theorem. Then, $A_n<A_\infty$ if and only if
\begin{equation}
	\left(
	1+\frac{1+\alpha}{n(1-\alpha)}
	\right)^2
	\cos^{1+\alpha}
	\left(
	\frac{\pi}{n(1-\alpha)+(1+\alpha)}
	\right)
	<1.
	\label{eq:An-less-Ainfty}
\end{equation}
In particular, $A_n>A_\infty\hbox{ for every integer }n\geq3$.
\end{coro}

\proof
To compare $A_n$ with $A_\infty$, divide the two explicit formulas and
raise to the power $1-\alpha$. This gives
\[
\left(\frac{A_n}{A_\infty}\right)^{1-\alpha}
=
\left(
1+\frac{1+\alpha}{n(1-\alpha)}
\right)^2
\cos^{1+\alpha}
\left(
\frac{\pi}{n(1-\alpha)+(1+\alpha)}
\right).
\]
Therefore $A_n<A_\infty$ if and only if \eqref{eq:An-less-Ainfty} holds.
This also shows that the convergence of $A_n$ towards $A_\infty$ is not
forced to be monotone: the algebraic prefactor and the cosine factor have
opposite effects as $n$ varies.	
From the above explicit formula
\[
R_n:=
\left(\frac{A_n}{A_\infty}\right)^{1-\alpha}
=
\left(
1+\frac{1+\alpha}{n(1-\alpha)}
\right)^2
\cos^{1+\alpha}
\left(
\frac{\pi}{n(1-\alpha)+(1+\alpha)}
\right).
\]
Clearly,
\[
R_n\longrightarrow1\qquad\hbox{as }n\to\infty.
\]
We shall prove that $R_n>1$ for every $n\geq3$.
Put
\[
\varepsilon:=1-\alpha\in(0,1),
\qquad
1+\alpha=2-\varepsilon,
\]
and regard $n$ temporarily as a real variable $s\geq3$. Define
\[
D(s):=s\varepsilon+2-\varepsilon,
\qquad
\theta(s):=\frac{\pi}{D(s)}.
\]
Then
\[
R(s)=
\left(\frac{D(s)}{s\varepsilon}\right)^2
\cos^{2-\varepsilon}\theta(s).
\]
A direct logarithmic differentiation gives
\[
\frac{R'(s)}{R(s)}
=
\frac{2-\varepsilon}{D(s)}
\left[
-\frac{2}{s}
+\varepsilon\theta(s)\tan\theta(s)
\right].
\]
Since
\[
s\varepsilon=\frac{\pi}{\theta}-(2-\varepsilon),
\]
the sign of $R'(s)$ is the sign of
\[
G(\theta(s))-2,
\qquad
G(\theta):=
[\pi-(2-\varepsilon)\theta]\tan\theta.
\]
We claim that $G$ is strictly increasing on $(0,\pi/2)$. Indeed,
\[
G'(\theta)
=
-(2-\varepsilon)\tan\theta
+
[\pi-(2-\varepsilon)\theta]\sec^2\theta,
\]
and therefore
\[
\cos^2\theta\,G'(\theta)
=
\pi-(2-\varepsilon)
\bigl(\theta+\sin\theta\cos\theta\bigr).
\]
The function
\[
h(\theta):=\theta+\sin\theta\cos\theta
\]
is strictly increasing on $(0,\pi/2)$, because
\[
h'(\theta)=2\cos^2\theta>0,
\]
and
\[
h(\theta)<h(\pi/2)=\frac{\pi}{2}.
\]
Consequently,
\[
\cos^2\theta\,G'(\theta)
>
\pi-(2-\varepsilon)\frac{\pi}{2}
=
\frac{\pi\varepsilon}{2}>0,
\]
which proves the claim.	
Now $\theta(s)$ is strictly decreasing in $s$. Since $G$ is strictly
increasing, the function $G(\theta(s))-2$ is strictly decreasing.
Hence $R'(s)$ can change sign at most once, and, if it does, it changes
from positive to negative. Thus $R$ is either decreasing on
$[3,\infty)$, or first increasing and then decreasing. In particular,
$R$ has no local minimum on $[3,\infty)$.	
Assume that for the case $n=3$ we have
\[
R(3)=
\left(\frac{A_3}{A_\infty}\right)^{1-\alpha}>1.
\]
Together with $\lim_{s\to\infty}R(s)=1$, the preceding implies
$R(s)>1\hbox{ for every finite }s\geq3$. Indeed, if $R$ is decreasing,
it decreases from $R(3)>1$ to its limit $1$; if it first increases and
then decreases, its decreasing part again approaches $1$ from above.
Therefore, for every integer $n\geq3$,
\[
\left(\frac{A_n}{A_\infty}\right)^{1-\alpha}>1.
\]	
Thus, we only need to prove that $R(3)>1$. We recall that
\begin{equation}
	\left(\frac{A_3}{A_\infty}\right)^{1-\alpha}
	=
	\left[
	\frac{2(2-\alpha)}{3(1-\alpha)}
	\right]^2
	\cos^{1+\alpha}
	\left(
	\frac{\pi}{2(2-\alpha)}
	\right).
	\label{eq:ratio-A3}
\end{equation}
Put $\varepsilon:=1-\alpha\in(0,1)$. Then
\eqref{eq:ratio-A3} can be rewritten as
\begin{equation}
	\left(\frac{A_3}{A_\infty}\right)^{1-\alpha}
	=
	\left[
	\frac{2(1+\varepsilon)}{3\varepsilon}
	\right]^2
	\sin^{2-\varepsilon}
	\left(
	\frac{\pi\varepsilon}{2(1+\varepsilon)}
	\right).
	\label{eq:ratio-epsilon}
\end{equation}
Indeed,
\[
\cos\left(\frac{\pi}{2(1+\varepsilon)}\right)
=
\sin\left(\frac{\pi\varepsilon}{2(1+\varepsilon)}\right).
\]
Introduce
\[
y:=\frac{\pi\varepsilon}{2(1+\varepsilon)}.
\]
Since $0<\varepsilon<1$,
\[
0<y<\frac{\pi}{4},
\]
and
\[
\frac{2(1+\varepsilon)}{\varepsilon}
=
\frac{\pi}{y}.
\]
Hence \eqref{eq:ratio-epsilon} takes the particularly simple form
\begin{equation}
	\left(\frac{A_3}{A_\infty}\right)^{1-\alpha}
	=
	\frac{\pi^2}{9y^2}\,
	\sin^{2-\varepsilon}y.
	\label{eq:ratio-y}
\end{equation}
We claim that
\begin{equation}
	\sin^{2-\varepsilon}y>y^2.
	\label{eq:key-ineq}
\end{equation}
Indeed,
\[
\sin y
=
y\,\frac{\sin y}{y},
\]
so \eqref{eq:key-ineq} is equivalent to
\[
\left(\frac{\sin y}{y}\right)^{2-\varepsilon}
y^{-\varepsilon}>1.
\]
Taking logarithms, it is enough to prove
\begin{equation}
	(2-\varepsilon)
	\log\frac{\sin y}{y}
	+
	\varepsilon\log\frac1y>0.
	\label{eq:log-ineq}
\end{equation}
From
\[
y=\frac{\pi\varepsilon}{2(1+\varepsilon)}
\]
we obtain
\begin{equation}
	\frac{\varepsilon}{2-\varepsilon}
	=
	\frac{y}{\pi-3y}.
	\label{eq:eps-y}
\end{equation}
On the other hand, for $0<y\leq\pi/4$,
\[
\frac{\sin y}{y}
\geq
1-\frac{y^2}{6}.
\]
Since $y^2/6<1/2$, we have
\begin{equation}
	-\log\frac{\sin y}{y}
	\leq
	-\log\left(1-\frac{y^2}{6}\right)
	\leq
	\frac{y^2}{3}.
	\label{eq:log-sin-est}
\end{equation}
We next observe that
\begin{equation}
	\frac{y^2}{3}
	<
	\frac{y}{\pi-3y}\log\frac1y,
	\qquad
	0<y\leq\frac{\pi}{4}.
	\label{eq:aux-ineq}
\end{equation}
Indeed, \eqref{eq:aux-ineq} is equivalent to
\[
\log\frac1y>
\frac{y(\pi-3y)}{3}.
\]
Define
\[
q(y):=
\log\frac1y-\frac{\pi y}{3}+y^2.
\]
Then
\[
q'(y)
=
-\frac1y-\frac{\pi}{3}+2y<0
\qquad
\hbox{for }0<y\leq\frac{\pi}{4}.
\]
Therefore $q$ is decreasing on this interval, and
\[
q(y)\geq q\left(\frac{\pi}{4}\right)
=
\log\frac4\pi-\frac{\pi^2}{48}>0.
\]
This proves \eqref{eq:aux-ineq}. Combining
\eqref{eq:log-sin-est}, \eqref{eq:aux-ineq}, and
\eqref{eq:eps-y}, we obtain
\[
-\log\frac{\sin y}{y}
<
\frac{\varepsilon}{2-\varepsilon}
\log\frac1y.
\]
Multiplication by $2-\varepsilon$ gives precisely
\[
(2-\varepsilon)
\log\frac{\sin y}{y}
+
\varepsilon\log\frac1y>0.
\]
Thus \eqref{eq:key-ineq} holds. Returning to \eqref{eq:ratio-y}, we
conclude that
\[
\left(\frac{A_3}{A_\infty}\right)^{1-\alpha}
>
\frac{\pi^2}{9}.
\]
Since $\frac{\pi^2}{9}>1$ and $1-\alpha>0$, it follows that
\[
A_3>A_\infty
\qquad\hbox{for every }0<\alpha<1.
\]
\fineq

\noindent
\begin{rem}
In fact, the preceding argument gives the stronger estimate
\[
\frac{A_3}{A_\infty}
>
\left(\frac{\pi^2}{9}\right)^{\frac{1}{1-\alpha}.
}
\]
Notice that the limiting case $\alpha\uparrow1$ is delicate,
since the two factors in \eqref{eq:ratio-A3} compensate each other. The
restriction $n\geq3$ is essential for this simple uniform statement. The
case $n=2$ is exceptional: depending on $\alpha$, the onset amplitude
$A_2$ may lie either above or below $A_\infty$.
\end{rem}

\begin{rem}[Why the degenerate profile may lie below the profile for $m\equiv-1$]
\label{rem:degenerate-versus-negative}

At first sight, the inequality between the amplitudes may seem
counter-intuitive. Indeed, since $m_n(x)\geq -1$, one might expect the
solution corresponding to the degenerate weight $m_n$ to be larger than the
solution for the strictly negative weight $m\equiv-1$: in the central region
$\Omega_n^0=(-1/n,1/n)$ the absorption term has disappeared completely.
This intuition would in fact be natural for a problem with a prescribed
positive source. For instance, consider
\[
-u''+a(x)u^\alpha=f(x),\qquad f(x)>0,
\]
with the same boundary conditions, and compare two non-negative
coefficients $a_1\leq a_2$. Under the usual assumptions ensuring the
comparison principle, the equation with the smaller coefficient has
less absorption, and one expects the corresponding solution to be
larger. In particular, replacing the coefficient $a\equiv1$ by a
coefficient which vanishes in a subinterval has the natural
interpretation
\[
\hbox{less absorption}\quad\Longrightarrow\quad
\hbox{larger solution}.
\]
The present problem is essentially different, because the right-hand
side is not prescribed. In the region where $m=-1$ the equation reads
\[
-u''+u^\alpha=\lambda u.
\]
Thus the term which plays the role of a source is $\lambda u$ itself.
Changing the solution changes simultaneously the right-hand side, so
the preceding comparison argument is no longer available. Equivalently,
the zero-order term $s\longmapsto s^\alpha-\lambda s$
is not globally monotone on $\mathbb R_+$, since
\[
\frac{d}{ds}\bigl(s^\alpha-\lambda s\bigr)
=
\alpha s^{\alpha-1}-\lambda
\]
changes sign. Hence the pointwise ordering $m_n\geq-1$ does not imply
the analogous ordering of the particular solutions selected on the
free-boundary branch.	
There is, moreover, a second and more precise explanation in the
present one-dimensional problem. For $m\equiv-1$, the zero-energy flat
orbit satisfies
\[
\frac12(u')^2
=
\frac{1}{1+\alpha}u^{1+\alpha}
-\frac{\lambda}{2}u^2,
\]
and its turning point has amplitude
\[
A_\infty(\lambda)
=
\left(
\frac{2}{(1+\alpha)\lambda}
\right)^{1/(1-\alpha)}.
\]
For the degenerate weight $m_n$, this nonlinear orbit is interrupted
at the interface $x=1/n$. Inside $\Omega_n^0$ the nonlinear absorption
term disappears and the even solution is necessarily
\[
u(x)=A_{\lambda,n}\cos(\sqrt{\lambda}\,x).
\]
The global solution is therefore not obtained merely by ``removing
some absorption'' from the solution for $m\equiv-1$; it must satisfy
the additional matching conditions at $x=1/n$. Matching the linear
core with the outer flat orbit gives
\[
A_{\lambda,n}
=
\left(
\frac{2}{(1+\alpha)\lambda}
\right)^{1/(1-\alpha)}
\cos^{(1+\alpha)/(1-\alpha)}
\left(\frac{\sqrt{\lambda}}{n}\right).
\]
Consequently, whenever
$0<\frac{\sqrt{\lambda}}{n}<\frac{\pi}{2}$, we have
\[
A_{\lambda,n}
=
A_\infty(\lambda)
\cos^{(1+\alpha)/(1-\alpha)}
\left(\frac{\sqrt{\lambda}}{n}\right)
<
A_\infty(\lambda).
\]

\noindent
Thus the apparently paradoxical inequality is not a direct monotonicity
effect with respect to the coefficient $m$. It is a global
free-boundary and matching effect. The zero region forces the
insertion of a linear oscillatory core, and the phase accumulated in
that core produces the factor
\[
\cos^{(1+\alpha)/(1-\alpha)}
\left(\frac{\sqrt{\lambda}}{n}\right)<1.
\]
In phase-plane language, the nonlinear zero-energy orbit for $m=-1$
is stopped before reaching its turning point and is joined to the
linear orbit generated inside $\Omega_n^0$. The sublinear character
$0<\alpha<1$ plays a double role. First, it is responsible for the
non-Lipschitz behavior at the origin which permits flat solutions and
free boundaries. Second, it amplifies the loss created by the matching,
because
\[
\frac{1+\alpha}{1-\alpha}>1,
\qquad
\frac{1+\alpha}{1-\alpha}\longrightarrow+\infty
\quad\hbox{as }\alpha\uparrow1.
\]
Thus even a modest phase loss in the cosine may produce a substantial
decrease of the amplitude. In the present equation, the term $\lambda u$
reverses the naive ordering suggested by the absorption coefficient.
\end{rem}

\begin{figure}[tph]
\begin{center}
\includegraphics[width=13.5cm]{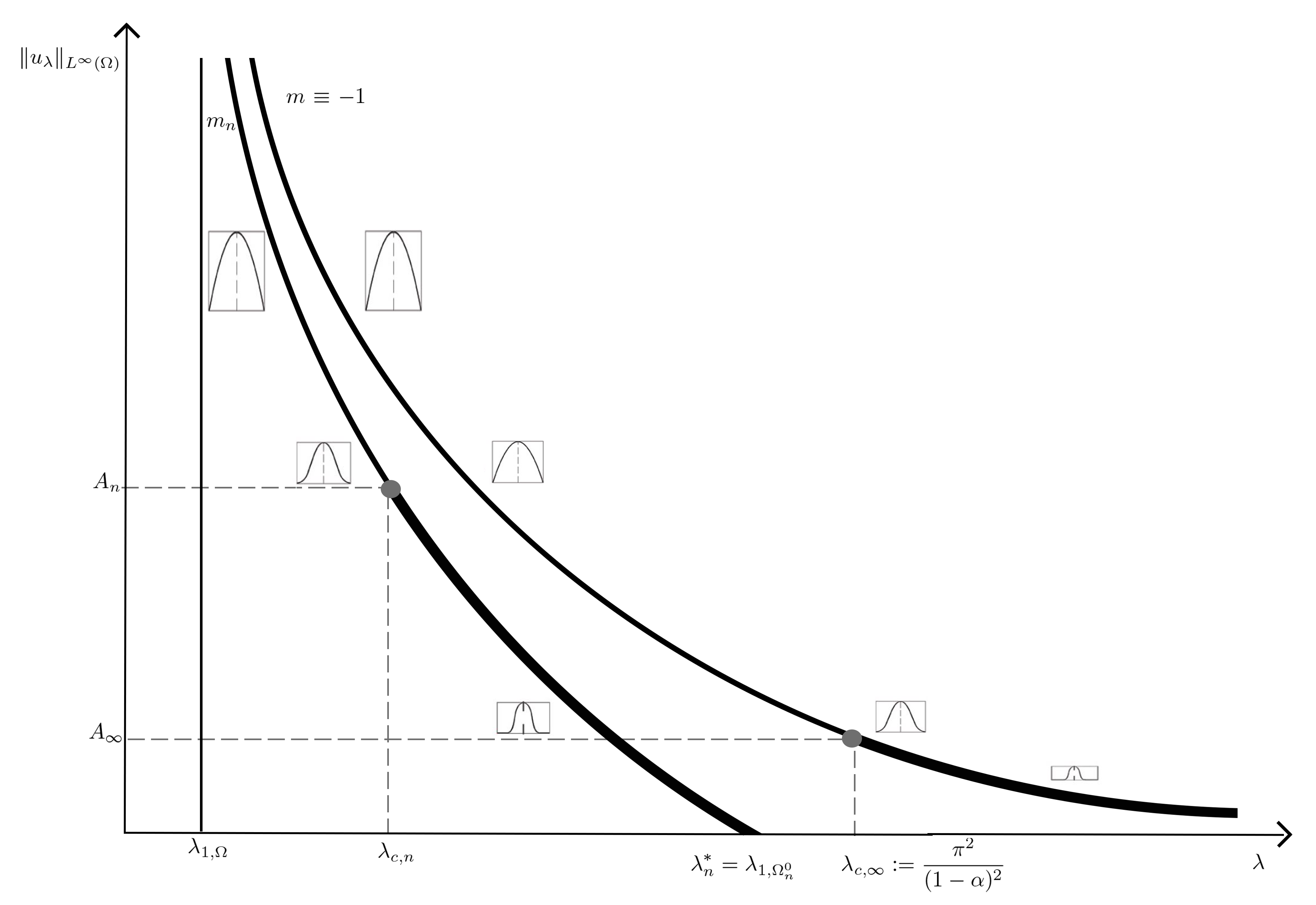}\\[0pt]
\end{center}
\caption{Qualitative behavior of the amplitude bifurcating curves for $m=-1$ and $m_n$ with $n>2$.}
\label{fig:energy-branches}
\end{figure}

\bigskip 

\subsection{\protect\bigskip Convergence of the supports of $u_{\protect%
\lambda }^{-}$ and of the free boundaries as $\protect\lambda \rightarrow 
\protect\lambda ^{\ast }=\protect\lambda _{1,\Omega _{0}}$}

We keep the notation and assumptions of Theorem \ref{thm:Main2}. In particular, 
$0<\alpha <1$,  $m\leq 0\quad \hbox{a.e. in }\Omega$, and $\Omega ^{0}$
is assumed to be a connected domain compactly contained in $\Omega $ such that $|\Omega^0|>0$. Moreover, $\lambda ^{\ast }=\lambda _{1,\Omega _{0}}$,
and $u_{\lambda }^{-}$ denotes the positive-energy branch constructed for $\lambda \in (\lambda _{1,\Omega },\lambda ^{\ast })$.
\begin{theo}
\label{thm:free-boundary-convergence}
Assume the hypotheses of Theorem~\ref{thm:Main2}.
Assume also that, for every $\varepsilon>0$ such that
\[
G_\varepsilon
:=
\{x\in\Omega:\operatorname{dist}(x,\Omega^0)>\varepsilon\}
\neq\varnothing,
\]
there exists a constant $q_\varepsilon>0$ satisfying
\begin{equation}
	-m(x)\geq q_\varepsilon
	\qquad\hbox{for a.e. }x\in G_\varepsilon.
	\label{eq:uniform-negativity-away}
\end{equation}
Set
\[
K_\lambda:=\operatorname{supp}u_\lambda^-,
\qquad
\Gamma_\lambda:=\partial K_\lambda\cap\Omega.
\]
Then, as $\lambda\uparrow\lambda^*=\lambda_{1,\Omega^0}$,
\begin{equation}
	d_H\bigl(K_\lambda,\overline{\Omega^0}\bigr)\longrightarrow0,
	\label{eq:support-Hausdorff}
\end{equation}
and
\begin{equation}
	d_H\bigl(\Gamma_\lambda,\partial\Omega^0\bigr)\longrightarrow0.
	\label{eq:free-boundary-Hausdorff}
\end{equation}
Here $d_H$ denotes the Hausdorff distance between compact subsets of
$\overline{\Omega}$.
\end{theo}

\proof
We divide the proof into three steps.

\medskip
\noindent
\textit{Step 1. The zero region $\Omega^0$ is contained in the positivity
	set of $u_\lambda^-$.}
By Theorem~\ref{thm:Main2},
\[
\frac{u_\lambda^-}
{\|u_\lambda^-\|_{H_0^1(\Omega)}}
\rightarrow
\widetilde{\varphi}_{1,\Omega^0}
\qquad\hbox{strongly in }H_0^1(\Omega),
\]
where $\widetilde{\varphi}_{1,\Omega^0}$ denotes the zero extension to
$\Omega$ of the first Dirichlet eigenfunction of $-\Delta$ in $\Omega^0$.
In particular, for $\lambda$ sufficiently close to $\lambda^*$, the
restriction of $u_\lambda^-$ to $\Omega^0$ is not identically zero.
Since $m=0$ a.e. in $\Omega^0$,
\[
-\Delta u_\lambda^-=\lambda u_\lambda^-
\qquad\hbox{in }\Omega^0.
\]
Therefore, by the strong maximum principle,
\[
u_\lambda^->0
\qquad\hbox{in }\Omega^0.
\]
Consequently, for every $\lambda$ sufficiently close to $\lambda^*$,
\begin{equation}
	\overline{\Omega^0}\subset K_\lambda.
	\label{eq:omega0-in-support}
\end{equation}

\medskip
\noindent
\textit{Step 2. The support cannot remain at a positive distance from
	$\Omega^0$.}
Fix $\varepsilon>0$ such that $G_\varepsilon\neq\emptyset$.
By \eqref{eq:uniform-negativity-away}, there exists $q_\varepsilon>0$
such that
\[
-m(x)\geq q_\varepsilon
\qquad\hbox{a.e. in }G_\varepsilon.
\]
Moreover, by Corollary~\ref{Coro L-infinity},
\[
\|u_\lambda^-\|_{L^\infty(\Omega)}
\rightarrow0
\qquad\hbox{as }\lambda\uparrow\lambda^*.
\]
Hence, for $\lambda$ sufficiently close to $\lambda^*$, all the
smallness assumptions required in the local dead-core criterion of
Theorem~\ref{thm:Comp} are satisfied in $G_\varepsilon$. Let
\[
R_{\lambda,\varepsilon}
:=
\psi_{1/N,q_\varepsilon}
\left(\|u_\lambda^-\|_{L^\infty(\Omega)}\right),
\]
where $\psi_{\mu,q}$ is the function introduced before
Theorem~\ref{thm:Comp}. Since
$\psi_{1/N,q_\varepsilon}(0)=0$, we have
\begin{equation}
	R_{\lambda,\varepsilon}\rightarrow0
	\qquad\hbox{as }\lambda\uparrow\lambda^*.
	\label{eq:R-goes-zero}
\end{equation}
Theorem~\ref{thm:Comp} gives
	\[
	K_\lambda\cap G_\varepsilon
	\subset
	\left\{
	x\in G_\varepsilon:
	\operatorname{dist}
	\bigl(x,\partial G_\varepsilon\setminus\partial\Omega\bigr)
	\leq
	R_{\lambda,\varepsilon}
	\right\}.
	\]
Since the interior boundary of $G_\varepsilon$ is contained in the
level set
\[
\{x\in\Omega:
\operatorname{dist}(x,\Omega^0)=\varepsilon\},
\]
it follows that
\begin{equation}
	K_\lambda
	\subset
		\left\{
		x\in\Omega:
		\operatorname{dist}(x,\Omega^0)
		\leq
		\varepsilon+R_{\lambda,\varepsilon}
		\right\}.
	\label{eq:outer-support-inclusion}
\end{equation}
Combining \eqref{eq:omega0-in-support} and
\eqref{eq:outer-support-inclusion}, we obtain
\[
d_H\bigl(K_\lambda,\overline{\Omega^0}\bigr)
\leq
\varepsilon+R_{\lambda,\varepsilon}.
\]
Taking the upper limit as $\lambda\uparrow\lambda^*$ and using
\eqref{eq:R-goes-zero}, we get
\[
\limsup_{\lambda\uparrow\lambda^*}
d_H\bigl(K_\lambda,\overline{\Omega^0}\bigr)
\leq\varepsilon.
\]
Since $\varepsilon>0$ is arbitrary, this proves
\eqref{eq:support-Hausdorff}.
Moreover, since $\Omega^0\Subset\Omega$, we may choose
	\[
	0<\varepsilon<
	\frac12\operatorname{dist}
	(\overline{\Omega^0},\partial\Omega).
	\]
	Then, by \eqref{eq:R-goes-zero} and
	\eqref{eq:outer-support-inclusion},
	$K_\lambda\Subset\Omega$ for all $\lambda$ sufficiently close to
	$\lambda^*$. In particular, $\Gamma_\lambda$ is compact for such
	$\lambda$.

\medskip
\noindent
\textit{Step 3. Convergence of the free boundaries.}
By Step~1,
\[
u_\lambda^->0
\qquad\hbox{in }\Omega^0,
\]
and hence
$\Gamma_\lambda\cap\Omega^0=\emptyset$.
Together with \eqref{eq:support-Hausdorff}, this yields
\begin{equation}
	\sup_{x\in\Gamma_\lambda}
	\operatorname{dist}(x,\partial\Omega^0)
	\rightarrow0.
	\label{eq:first-semidistance}
\end{equation}

It remains to prove the reverse semidistance. Suppose, by contradiction,
that there exist $\eta>0$, a sequence
$\lambda_n\uparrow\lambda^*$, and points
$y_n\in\partial\Omega^0$ such that
\begin{equation}
	\operatorname{dist}(y_n,\Gamma_{\lambda_n})\geq\eta
	\qquad\hbox{for every }n.
	\label{eq:contradiction-free-boundary}
\end{equation}
Passing to a subsequence, we may assume that
\[
y_n\rightarrow y\in\partial\Omega^0.
\]
Choose
	\[
	0<r<
	\min\left\{
	\frac{\eta}{4},
	\frac14\operatorname{dist}
	(\overline{\Omega^0},\partial\Omega)
	\right\}.
	\]
Since $\partial\Omega^0$ is of class $C^{1,1}$, it satisfies a uniform
exterior ball condition. Hence, for all sufficiently large $n$, there
exist points
\[
z_n\in\Omega\setminus\overline{\Omega^0},
\qquad
|z_n-y_n|<r,
\]
such that
\[
\operatorname{dist}(z_n,\Omega^0)\geq\frac r2.
\]
By the already proved Hausdorff convergence
\eqref{eq:support-Hausdorff},
\[
z_n\notin K_{\lambda_n}
\]
for all sufficiently large $n$.
On the other hand, choose
\[
x_n\in\Omega^0,
\qquad
|x_n-y_n|<r.
\]
By Step~1, $x_n\in\operatorname{int}K_{\lambda_n}$.
Since $z_n\notin K_{\lambda_n}$, the line segment joining $x_n$ to
$z_n$ must meet $\partial K_{\lambda_n}$. By the choice of $r$, this
segment is contained in $\Omega$, and therefore it meets
\[
\partial K_{\lambda_n}\cap\Omega
=
\Gamma_{\lambda_n}
\]
at some point $\xi_n$. Since both $x_n$ and $z_n$ belong to $B_r(y_n)$ and this ball is
	convex,
	\[
	|\xi_n-y_n|<r<\eta.
	\]
Thus
\[
\operatorname{dist}(y_n,\Gamma_{\lambda_n})<\eta,
\]
contradicting \eqref{eq:contradiction-free-boundary}. Therefore
\[
\sup_{y\in\partial\Omega^0}
\operatorname{dist}(y,\Gamma_\lambda)
\longrightarrow0.
\]
Together with \eqref{eq:first-semidistance}, this proves
\eqref{eq:free-boundary-Hausdorff}.$\fin$

\begin{rem}[A quantitative version]
\label{rem:quantitative-free-boundary} Assume, more strongly, that there
exists $q_{\ast }>0$ such that 
\begin{equation*}
-m(x)\geq q_{\ast }\qquad \hbox{for a.e. }x\in \Omega \setminus \overline{%
\Omega ^{0}}.
\end{equation*}%
Then Theorem \ref{thm:Comp} may be applied directly to $G=\Omega \setminus \overline{\Omega ^{0}}$.
Hence 
\begin{equation*}
K_{\lambda }\setminus \overline{\Omega ^{0}}\subset \left\{ x\in \Omega
\setminus \overline{\Omega ^{0}}:\hbox{dist}(x,\partial \Omega
^{0})<R_{\lambda }\right\} ,
\end{equation*}%
where $R_{\lambda }=\psi _{1/N,q_{\ast }}\left( \Vert u_{\lambda }^{-}\Vert
_{L^{\infty }(\Omega )}\right)$. Since, as $s\downarrow 0$, 
\begin{equation*}
\psi _{1/N,q_{\ast }}(s)\sim C_{N,\alpha ,q_{\ast }}\,s^{(1-\alpha )/2},
\end{equation*}%
one obtains the quantitative localization estimate 
\begin{equation*}
\sup_{x\in K_{\lambda }}\hbox{dist}(x,\overline{\Omega ^{0}})\leq
C_{N,\alpha ,q_{\ast }}\Vert u_{\lambda }^{-}\Vert _{L^{\infty }(\Omega
)}^{(1-\alpha )/2}
\end{equation*}%
for $\lambda $ sufficiently close to $\lambda ^{\ast }$. In particular, the
same rate controls the outer semidistance of the free boundary from $%
\partial \Omega ^{0}$.
\end{rem}

\subsection{Localization of the support of $u_{\protect\lambda}^{+}$ as
$\protect\lambda\rightarrow-\infty$}

\begin{theo}
\label{thm:free-boundary-convergence-plus}
Let $0<\alpha<1$, let $\Omega\subset\mathbb R^N$ be bounded and connected,
and let $m\in L^\infty(\Omega)$.
Set $\Omega^+=\{m>0\}$ and assume that $\Omega^+\Subset\Omega$ is nonempty
and connected, $\partial\Omega^+\in C^{1,1}$, and, for some
$m_0,q_0>0$,
\[
m\geq m_0>0
\quad\hbox{a.e. in }\Omega^+,
\qquad
m\leq-q_0<0
\quad\hbox{a.e. in }\Omega\setminus\overline{\Omega^+}.
\]
Let $u_\lambda^+$ be the non-negative solution of \eqref{Eq1} given by
	Theorem~\ref{thm1} for $\lambda<0$, and put
\[
K_\lambda=\operatorname{supp}u_\lambda^+,
\qquad
\Gamma_\lambda=\partial K_\lambda\cap\Omega.
\]
Then there exists $C>0$ such that, for all sufficiently negative $\lambda$,
\begin{equation}
	\overline{\Omega^+}
	\subset K_\lambda
	\subset
	\left\{
	x\in\Omega:
	\operatorname{dist}(x,\overline{\Omega^+})
	<C|\lambda|^{-1/2}
	\right\}.
	\label{Conclus1}
\end{equation}
Consequently,
\begin{equation}
	d_H(K_\lambda,\overline{\Omega^+})
	\leq C|\lambda|^{-1/2}
	\rightarrow0,
	\label{Conclusion2}
\end{equation}
and, possibly enlarging $C$,
\begin{equation}
	d_H(\Gamma_\lambda,\partial\Omega^+)
	\leq C|\lambda|^{-1/2}
	\rightarrow0.
	\label{Conclusion3}
\end{equation}
\end{theo}

\bigskip

Before presenting the proof, it is useful to obtain a comparison result for
solutions to this non-monotone elliptic equation when the boundary traces
are ordered.

\begin{lemma}
\label{lem:one-sided-general-comparison}
Let $D\subset\mathbb R^N$ be a bounded Lipschitz domain and let
$F:D\times(0,+\infty)\to\mathbb R$ be a Carath\'eodory function such that,
for a.e. $x\in D$, the map
\begin{equation}
	s\mapsto\frac{F(x,s)}{s}
	\label{eq:structural-condition}
\end{equation}
is strictly decreasing on $(0,+\infty)$. Let
$U,V\in H^1(D)\cap L^\infty(D)$, with $U>0$, $V>0$ a.e. in $D$, and assume
the one-sided trace condition
\begin{equation}
	(U-V)^+\in H_0^1(D).
	\label{eq:trace-condition}
\end{equation}
Suppose that $U$ is a weak subsolution and $V$ a weak supersolution of
\begin{equation}
	-\Delta w=F(x,w)
	\qquad\hbox{in }D,
	\label{eq:general-equation}
\end{equation}
that is,
\begin{align}
	\int_D\nabla U\cdot\nabla\varphi\,dx
	&\leq
	\int_DF(x,U)\varphi\,dx,
	\label{eq:subsolution}\\
	\int_D\nabla V\cdot\nabla\varphi\,dx
	&\geq
	\int_DF(x,V)\varphi\,dx
	\label{eq:supersolution}
\end{align}
for every non-negative
$\varphi\in H_0^1(D)\cap L^\infty(D)$ for which the right-hand sides are
finite. Assume, moreover, that
\begin{equation}
	F(\cdot,U),F(\cdot,V)\in L^1_{\rm loc}(D)
	\label{eq:local-integrability}
\end{equation}
and
\begin{equation}
	\left(
	\frac{|F(x,U)|}{U}
	+
	\frac{|F(x,V)|}{V}
	\right)
	(U^2-V^2)^+
	\in L^1(D).
	\label{eq:integrability-assumption}
\end{equation}
Then
\[
U\leq V
\qquad\hbox{a.e. in }D.
\]
\end{lemma}

\proof
Set $E:=\{x\in D:U(x)>V(x)\}$. We will prove that $|E|=0$.
We use the pointwise algebraic identity associated with the works of
Brezis and Oswald \cite{Brezis-Oswald}, D\'iaz and Sa\'a
\cite{DiazSaa}, and, in a pioneering form and a different framework,
Picone \cite{Picone}
\begin{align}
	&\nabla U\cdot
	\nabla\left(\frac{U^2-V^2}{U}\right)
	-
	\nabla V\cdot
	\nabla\left(\frac{U^2-V^2}{V}\right)
	\nonumber\\
	&\quad=
	\left|\nabla U-\frac{U}{V}\nabla V\right|^2
	+
	\left|\nabla V-\frac{V}{U}\nabla U\right|^2
	\geq0.
	\label{eq:picone-general}
\end{align}
This formal computation is classical, so we do not reproduce its
derivation. The delicate point here is that the two functions may have
different boundary traces and the quotients $1/U$ and $1/V$ may become
singular. We therefore give the regularization and passage to the limit.
For $\varepsilon>0$, set
\[
U_\varepsilon:=U+\varepsilon,
\qquad
V_\varepsilon:=V+\varepsilon,
\]
and define
\begin{equation}
	W_\varepsilon
	:=
	(U_\varepsilon^2-V_\varepsilon^2)^+.
	\label{eq:Weps}
\end{equation}
Since $U_\varepsilon-V_\varepsilon=U-V$,
\[
\{U_\varepsilon>V_\varepsilon\}=E.
\]
Moreover, on $E$,
\begin{equation}
	W_\varepsilon
	=
	(U-V)(U+V+2\varepsilon)
	=
	(U^2-V^2)+2\varepsilon(U-V).
	\label{eq:Weps-on-E}
\end{equation}
Because $U,V\in H^1(D)\cap L^\infty(D)$,
$U_\varepsilon^2,V_\varepsilon^2\in H^1(D)$.
Condition \eqref{eq:trace-condition} means that the trace of $U$ does not
exceed that of $V$. Hence the trace of $W_\varepsilon$ is zero, and
$W_\varepsilon\in H_0^1(D)$. We may therefore introduce
\begin{equation}
	\varphi_\varepsilon
	:=
	\frac{W_\varepsilon}{U_\varepsilon},
	\qquad
	\psi_\varepsilon
	:=
	\frac{W_\varepsilon}{V_\varepsilon}.
	\label{eq:regularized-tests}
\end{equation}
Since $U_\varepsilon,V_\varepsilon\geq\varepsilon$, both are non-negative
elements of $H_0^1(D)\cap L^\infty(D)$ and are admissible in
\eqref{eq:subsolution}--\eqref{eq:supersolution}. Using
$\varphi_\varepsilon$ in the subsolution inequality and
$\psi_\varepsilon$ in the supersolution inequality, and subtracting, we
obtain
\begin{align}
	I_\varepsilon
	:={}&
	\int_D
	\left\{
	\nabla U\cdot\nabla\varphi_\varepsilon
	-
	\nabla V\cdot\nabla\psi_\varepsilon
	\right\}\,dx
	\nonumber\\
	\leq{}&
	\int_D
	\left[
	\frac{F(x,U)}{U_\varepsilon}
	-
	\frac{F(x,V)}{V_\varepsilon}
	\right]
	W_\varepsilon\,dx.
	\label{eq:regularized-integral-inequality}
\end{align}
Since $\nabla U_\varepsilon=\nabla U$ and
$\nabla V_\varepsilon=\nabla V$, identity
\eqref{eq:picone-general}, applied to
$(U_\varepsilon,V_\varepsilon)$, gives on $E$
\begin{align}
	&\nabla U\cdot\nabla\varphi_\varepsilon
	-
	\nabla V\cdot\nabla\psi_\varepsilon
	\nonumber\\
	&\quad=
	\left|
	\nabla U-\frac{U_\varepsilon}{V_\varepsilon}\nabla V
	\right|^2
	+
	\left|
	\nabla V-\frac{V_\varepsilon}{U_\varepsilon}\nabla U
	\right|^2
	\geq0,
	\label{eq:regularized-picone}
\end{align}
whereas both sides vanish outside $E$. Hence
\begin{equation}
	I_\varepsilon\geq0.
	\label{eq:Ieps-positive}
\end{equation}
We now pass to the limit $\varepsilon\downarrow0$. For a.e. $x\in E$,
since $U(x),V(x)>0$,
\[
\frac{U_\varepsilon}{V_\varepsilon}\to\frac UV,
\qquad
\frac{V_\varepsilon}{U_\varepsilon}\to\frac VU.
\]
The integrand in \eqref{eq:regularized-picone} is non-negative.
Therefore Fatou's lemma gives
\begin{align}
	&\int_E
	\left[
	\left|\nabla U-\frac UV\nabla V\right|^2
	+
	\left|\nabla V-\frac VU\nabla U\right|^2
	\right]dx
	\nonumber\\
	&\qquad\leq
	\liminf_{\varepsilon\downarrow0}I_\varepsilon.
	\label{eq:fatou-picone}
\end{align}
This is why no a priori bound on $U/V$ or $V/U$ is needed
(in contrast to \cite{Brezis-Oswald}).	
It remains to pass to the limit in the right-hand side of
\eqref{eq:regularized-integral-inequality}. On $E$, by
\eqref{eq:Weps-on-E},
\[
W_\varepsilon=W+2\varepsilon(U-V),
\qquad
W:=(U^2-V^2)^+.
\]
For the first term,
\begin{align}
	\left|
	\frac{F(x,U)}{U_\varepsilon}W_\varepsilon
	\right|
	&\leq
	\frac{|F(x,U)|}{U}\,W
	+
	2|F(x,U)|(U-V)
	\nonumber\\
	&\leq
	3\,\frac{|F(x,U)|}{U}\,W,
	\label{eq:dominate-U}
\end{align}
because
\[
U(U-V)\leq(U-V)(U+V)=W
\qquad\hbox{on }E.
\]
Similarly,
\begin{align}
	\left|
	\frac{F(x,V)}{V_\varepsilon}W_\varepsilon
	\right|
	&\leq
	\frac{|F(x,V)|}{V}\,W
	+
	2|F(x,V)|(U-V)
	\nonumber\\
	&\leq
	3\,\frac{|F(x,V)|}{V}\,W,
	\label{eq:dominate-V}
\end{align}
because $V(U-V)\leq W$ on $E$.
The right-hand sides of \eqref{eq:dominate-U} and
\eqref{eq:dominate-V} are integrable by
\eqref{eq:integrability-assumption}. Thus dominated convergence yields
\begin{align}
	&\lim_{\varepsilon\downarrow0}
	\int_D
	\left[
	\frac{F(x,U)}{U_\varepsilon}
	-
	\frac{F(x,V)}{V_\varepsilon}
	\right]
	W_\varepsilon\,dx
	\nonumber\\
	&\qquad=
	\int_E
	\left[
	\frac{F(x,U)}{U}
	-
	\frac{F(x,V)}{V}
	\right]
	(U^2-V^2)\,dx.
	\label{eq:reaction-limit}
\end{align}
Combining \eqref{eq:regularized-integral-inequality},
\eqref{eq:fatou-picone}, and \eqref{eq:reaction-limit}, we obtain the
one-sided D\'iaz--Sa\'a inequality \cite{DiazSaa}
(see also \cite{Diaz PAFA}):
\begin{align}
	0
	&\leq
	\int_E
	\left[
	\left|\nabla U-\frac UV\nabla V\right|^2
	+
	\left|\nabla V-\frac VU\nabla U\right|^2
	\right]dx
	\nonumber\\
	&\leq
	\int_E
	\left[
	\frac{F(x,U)}{U}
	-
	\frac{F(x,V)}{V}
	\right]
	(U^2-V^2)\,dx.
	\label{eq:final-one-sided-integral}
\end{align}
Finally, on $E$, $U>V>0$. By the strict monotonicity assumption
\eqref{eq:structural-condition},
\[
\frac{F(x,U)}{U}
<
\frac{F(x,V)}{V}
\qquad\hbox{for a.e. }x\in E,
\]
while $U^2-V^2>0$. Hence the last integral in
\eqref{eq:final-one-sided-integral} is strictly negative whenever
$|E|>0$, contradicting its non-negativity. Therefore $|E|=0$, and
$U\leq V$ a.e. in $D.\fin$

\begin{rem}
\label{rem:BO-different-traces}
The structural condition \eqref{eq:structural-condition} is the classical
Brezis--Oswald subhomogeneity condition. The point of
Lemma~\ref{lem:one-sided-general-comparison} is that the two functions need
not have the same Dirichlet trace; the only boundary assumption is the
ordered-trace condition
$(U-V)^+\in H_0^1(D)$. For this reason, the usual symmetric
Brezis--Oswald identity (see expression (10) in \cite{Brezis-Oswald})
cannot be invoked directly. If, for instance, $U=0$ and $V\geq0$ on
$\partial D$, with $V$ positive on a portion of the boundary, the quotient
$V^2/U$ may be singular there and need not be an admissible test function.
The regularization
\[
U_\varepsilon=U+\varepsilon,
\qquad
V_\varepsilon=V+\varepsilon,
\qquad
W_\varepsilon=(U_\varepsilon^2-V_\varepsilon^2)^+
\]
avoids this difficulty and tests only the region where the desired ordering
could fail. Thus \eqref{eq:final-one-sided-integral} is the appropriate
one-sided version of the Brezis--Oswald/D\'iaz--Sa\'a/Picone mechanism for
ordered, possibly different, boundary traces.
\end{rem}

\par
\noindent
{\sc Proof of Theorem~\ref{thm:free-boundary-convergence-plus}.}
For $\lambda<0$, let $U_\lambda$ be the unique positive solution of
\[
-\Delta U+|\lambda|U=m(x)U^\alpha
\quad\hbox{in }\Omega^+,
\qquad
U=0
\quad\hbox{on }\partial\Omega^+
\]
(see, e.g., \cite{Diaz PAFA}). The restriction of $u_\lambda^+$ to
$\Omega^+$ satisfies the same equation and has non-negative boundary trace.
We first note that $u_\lambda^+>0$ in $\Omega^+$. Indeed,
$M(u_\lambda^+)>0$, and hence $u_\lambda^+$ is not identically zero in
$\Omega^+$. Moreover,
\[
-\Delta u_\lambda^+ +|\lambda|u_\lambda^+
=
m(x)(u_\lambda^+)^\alpha\geq0
\qquad\hbox{in }\Omega^+.
\]
Since $\Omega^+$ is connected, the strong maximum principle yields
\[
u_\lambda^+>0
\qquad\hbox{in }\Omega^+.
\]	
To apply Lemma~\ref{lem:one-sided-general-comparison}, set
\[
F(x,s):=-|\lambda|s+m(x)s^\alpha.
\]
Then
\[
\frac{F(x,s)}{s}
=
-|\lambda|+m(x)s^{\alpha-1}
\]
is strictly decreasing in $s>0$ for a.e. $x\in\Omega^+$, since
$m(x)\geq m_0>0$ and $0<\alpha<1$. Moreover,
$U_\lambda=0$ on $\partial\Omega^+$ and the trace of $u_\lambda^+$ is
non-negative, so
\[
(U_\lambda-u_\lambda^+)^+\in H_0^1(\Omega^+).
\]	
It remains to verify \eqref{eq:integrability-assumption}. Put
\[
E_\lambda
:=
\{x\in\Omega^+:U_\lambda(x)>u_\lambda^+(x)\}.
\]
On $E_\lambda$,
\[
\frac{|F(x,U_\lambda)|}{U_\lambda}
\leq
|\lambda|+\|m\|_\infty U_\lambda^{\alpha-1}.
\]
Since
\[
\bigl(U_\lambda^2-(u_\lambda^+)^2\bigr)^+
\leq U_\lambda^2,
\]
we obtain
\[
\frac{|F(x,U_\lambda)|}{U_\lambda}
\bigl(U_\lambda^2-(u_\lambda^+)^2\bigr)^+
\leq
|\lambda|U_\lambda^2
+
\|m\|_\infty U_\lambda^{\alpha+1},
\]
which belongs to $L^1(\Omega^+)$.	
For the term involving $u_\lambda^+$, we claim that
$U_\lambda/u_\lambda^+$ is bounded on $E_\lambda$.
Away from $\partial\Omega^+$ this follows from the positivity and
continuity of $u_\lambda^+$. Near a boundary point at which the trace of
$u_\lambda^+$ is positive, the set $E_\lambda$ is empty in a sufficiently
small neighborhood, since $U_\lambda=0$ on $\partial\Omega^+$. At a
boundary point at which the trace of $u_\lambda^+$ vanishes, the Hopf
boundary lemma, applied to
\[
-\Delta u_\lambda^+ +|\lambda|u_\lambda^+
=
m(x)(u_\lambda^+)^\alpha>0
\qquad\hbox{in }\Omega^+,
\]
gives a linear lower bound for $u_\lambda^+$ in terms of
$\operatorname{dist}(x,\partial\Omega^+)$, whereas
$U_\lambda\in C^1(\overline{\Omega^+})$ and
$U_\lambda=0$ on $\partial\Omega^+$ give a corresponding linear upper
bound. By compactness of the zero-trace part of $\partial\Omega^+$,
a finite covering yields
\[
U_\lambda\leq C_\lambda u_\lambda^+
\qquad\hbox{on }E_\lambda
\]
for some $C_\lambda>0$.
Consequently,
\[
\begin{aligned}
	&\frac{|F(x,u_\lambda^+)|}{u_\lambda^+}
	\bigl(U_\lambda^2-(u_\lambda^+)^2\bigr)^+
	\\
	&\qquad\leq
	|\lambda|U_\lambda^2
	+
	\|m\|_\infty
	(u_\lambda^+)^{\alpha-1}U_\lambda^2
	\\
	&\qquad\leq
	|\lambda|U_\lambda^2
	+
	C_\lambda^2\|m\|_\infty
	(u_\lambda^+)^{\alpha+1},
\end{aligned}
\]
which is integrable in $\Omega^+$. Thus
\eqref{eq:integrability-assumption} holds.
Therefore Lemma~\ref{lem:one-sided-general-comparison}, with
\[
D=\Omega^+,\qquad
U=U_\lambda,\qquad
V=u_\lambda^+|_{\Omega^+},
\]
gives
\[
u_\lambda^+\geq U_\lambda>0
\qquad\hbox{in }\Omega^+.
\]
Hence
\begin{equation}
\overline{\Omega^+}\subset K_\lambda.
\label{Estim7}
\end{equation}
Set $G=\Omega\setminus\overline{\Omega^+}$. By the assumptions of the
theorem,
\[
-m(x)\geq q_0
\qquad\hbox{for a.e. }x\in G.
\]
Lemma~\ref{lem:Linfty-negative-lambda} gives
\begin{equation}
\|u_\lambda^+\|_\infty
\leq
\left(
\frac{\|m^+\|_\infty}{|\lambda|}
\right)^{1/(1-\alpha)}.
\label{estim8}
\end{equation}
Applying Theorem~\ref{thm:Comp} in $G$ yields
\[
K_\lambda\cap G
\subset
\left\{
x\in G:
\operatorname{dist}(x,\partial\Omega^+)\leq R_\lambda
\right\},
\]
where
\[
R_\lambda
=
\psi_{1/N,q_0}(\|u_\lambda^+\|_\infty).
\]
Since
\[
\psi_{1/N,q_0}(\tau)
\leq
\frac{\sqrt{2N(\alpha+1)/q_0}}{1-\alpha}
\,\tau^{(1-\alpha)/2},
\]
\eqref{estim8} gives
\begin{equation}
R_\lambda\leq C_0|\lambda|^{-1/2}.
\label{9}
\end{equation}
Choose a fixed $C>C_0$. Then, for all sufficiently negative $\lambda$,
\[
K_\lambda
\subset
\left\{
x\in\Omega:
\operatorname{dist}(x,\overline{\Omega^+})
<
C|\lambda|^{-1/2}
\right\}.
\]
Together with \eqref{Estim7}, this proves \eqref{Conclus1} and
\eqref{Conclusion2}.

Put
\[
\varepsilon_\lambda=C|\lambda|^{-1/2}.
\]
Then
\[
\overline{\Omega^+}
\subset K_\lambda
\subset
(\overline{\Omega^+})_{\varepsilon_\lambda},
\]
where
\[
(\overline{\Omega^+})_{\varepsilon_\lambda}
:=
\left\{
x\in\Omega:
\operatorname{dist}(x,\overline{\Omega^+})
<\varepsilon_\lambda
\right\}.
\]
Since $u_\lambda^+>0$ in $\Omega^+$, no point of $\Gamma_\lambda$ lies in
$\Omega^+$, and for $x\in\Gamma_\lambda$,
\[
\operatorname{dist}(x,\overline{\Omega^+})
=
\operatorname{dist}(x,\partial\Omega^+).
\]
Hence
\begin{equation}
\sup_{x\in\Gamma_\lambda}
\operatorname{dist}(x,\partial\Omega^+)
\leq
\varepsilon_\lambda.
\label{Estim10}
\end{equation}
Conversely, let $y\in\partial\Omega^+$ and let $\nu(y)$ denote the exterior
unit normal. Since $\partial\Omega^+\in C^{1,1}$, there exists a uniform
tubular neighborhood of $\partial\Omega^+$. Moreover,
$\Omega^+\Subset\Omega$. Hence, for all sufficiently negative $\lambda$,
\[
z_\lambda
:=
y+2\varepsilon_\lambda\nu(y)
\in\Omega
\]
and
\[
\operatorname{dist}
(z_\lambda,\overline{\Omega^+})
=
2\varepsilon_\lambda.
\]
Thus $z_\lambda\notin K_\lambda$, whereas $y\in K_\lambda$ by
\eqref{Estim7}. If $y\in\Gamma_\lambda$, there is nothing to prove.
Otherwise, the segment joining $y$ to $z_\lambda$ must leave the closed set
$K_\lambda$ at some point
$\xi_\lambda\in\Gamma_\lambda$, and
\[
|\xi_\lambda-y|\leq2\varepsilon_\lambda.
\]
Therefore
\begin{equation}
\sup_{y\in\partial\Omega^+}
\operatorname{dist}(y,\Gamma_\lambda)
\leq
2\varepsilon_\lambda.
\label{Estim11}
\end{equation}
Equations \eqref{Estim10}--\eqref{Estim11} prove
\eqref{Conclusion3}.$\fin$

\begin{rem}
Setting
\[
\mu=|\lambda|=-\lambda,
\qquad
v_\mu=\mu^{1/(1-\alpha)}u_{-\mu}^+,
\]
the equation in $\Omega^+$ becomes exactly
\begin{equation}
	-\frac1\mu\Delta v_\mu+v_\mu=m(x)v_\mu^\alpha.
	\label{4}
\end{equation}
Thus, away from $\partial\Omega^+$, the formal dominant balance is
$v=m(x)v^\alpha$, so that on the positive branch
\begin{equation}
	v=m(x)^{1/(1-\alpha)}.
	\label{5}
\end{equation}
Equivalently, the expected leading-order interior behavior is
\begin{equation}
	u_\lambda^+(x)
	\sim
	|\lambda|^{-1/(1-\alpha)}
	m(x)^{1/(1-\alpha)}.
	\label{6}
\end{equation}
The same rescaling shows that the natural transition-layer thickness is
$\mu^{-1/2}=|\lambda|^{-1/2}$, exactly the scale appearing in
\eqref{Conclus1}--\eqref{Conclusion3}. Indeed, write
$\lambda=-\mu$, $\mu\to+\infty$, and set
\[
u_{-\mu}^+=\mu^{-1/(1-\alpha)}v_\mu.
\]
Substitution into
\[
-\Delta u_{-\mu}^+
+\mu u_{-\mu}^+
=
m(x)(u_{-\mu}^+)^\alpha
\]
gives the exact equation \eqref{4}. The zero-order and nonlinear terms
therefore have the same order, whereas diffusion carries the small
coefficient $1/\mu$. Formally letting $\mu\to+\infty$ on compact subsets of
$\Omega^+$ gives \eqref{5}, and hence \eqref{6}. Finally, writing
$\varepsilon^2=1/\mu$ gives
\[
-\varepsilon^2\Delta v_\mu+v_\mu=m(x)v_\mu^\alpha,
\]
so the natural length scale on which diffusion competes with the remaining
terms is
\[
\varepsilon
=
\mu^{-1/2}
=
|\lambda|^{-1/2}.
\]
This coincides with the rigorous localization scale. Note that the support
and free-boundary convergence in
\eqref{Conclusion2}--\eqref{Conclusion3}, including the rate
$O(|\lambda|^{-1/2})$, are rigorous consequences of
Lemma~\ref{lem:Linfty-negative-lambda}, Theorem~\ref{thm:Comp}, and the
interior comparison. Formulae \eqref{4}--\eqref{6} explain the dominant
balance and the natural boundary-layer scale.
\end{rem}
\noindent
{\bf \sffamily Acknowledgements} JID was partially supported by the
projects PID-2020-112517GBI00 of the AEI and PID2023-146754NB-I00 funded by
MCIU/AEI/10.13039/501100011033 and FEDER, EU.
MCIU/AEI/10.13039/-501100011033/FEDER, EU.

\end{document}